\RequirePackage{fix-cm}
\documentclass[smallextended]{svjour3}       
\smartqed  
\usepackage{graphicx}

\usepackage{amsfonts,amssymb,amsmath,here}

\newtheorem{main}{Main Theorem}

\makeatletter

\@addtoreset{equation}{section}
\makeatother

\def \calL{{\mathcal L}}

\def \calF{{\mathcal F}}
\def \calM{{\mathcal M}}

\def \N{{\mathbb N}}
\def \Z{{\mathbb Z}}
\def \Q{{\mathbb Q}}
\def \R{{\mathbb R}}

\def \F{{\calF}}
\def \L{{\calL}}
\def \lcm{\mathrm{lcm}}

\def \Gal{\mathrm{Gal}}
\def \Tr{\mathrm{Tr}}

\def \mod#1{\ ({\rm mod}\ {#1})}
\def \pmod#1{\mod{#1}}

\begin{document}

\title{A family of pairs of imaginary cyclic fields 
of degree $(p-1)/2$ with both class numbers divisible by $p$\thanks{This work was supported by JSPS KAKENHI Grant Numbers JP26400015 and JP15K04779.}}

\titlerunning{A family of pairs of imaginary cyclic fields}       

\author{Miho Aoki         \and
        Yasuhiro Kishi 
}

\institute{ Miho Aoki \at
Department of Mathematics,
Interdisciplinary Faculty of Science and Engineering,
Shimane University,
Matsue, Shimane, 690-8504, Japan \\
              \email{aoki@riko.shimane-u.ac.jp}          
           \and
           Yasuhiro Kishi \at
Department of Mathematics, 
Faculty of Education,
Aichi University of Education, 
Kariya, Aichi, 448-8542, Japan   \\
 \email{ykishi@auecc.aichi-edu.ac.jp}       
}

\date{Received: date / Accepted: date}

\maketitle

\begin{abstract}
Let $p$ be a prime number with $p\equiv 5\mod{8}$.
We construct a new infinite  family of pairs of
imaginary cyclic  fields of degree $(p-1)/2$ with both class numbers divisible by $p$.
Let $k_0$ be the unique subfield of $\Q (\zeta_p)$ of degree $(p-1)/4$ and 
$u_p=(t+b\sqrt{p})/2\,(>1)$ be the fundamental unit of $k:=\Q(\sqrt{p})$.
We put 
$D_{m,n}:=\L_m(2\F_m-\F_n\L_m)b$ for integers $m$ and $n$, where 
$\{ \F_n \}$ and $\{ \L_n \}$ are linear recurrence sequences of degree two associated to the
characteristic polynomial $P(X)=X^2-tX-1$.
We assume that there exists a pair $(m_0,n_0)$ of integers satisfying certain congruence relations. Then we show that 
there exists a positive integer $N_q$ which satisfies the both class numbers of $k_0(\sqrt{D_{m,n}})$ and $k_0(\sqrt{pD_{m,n}})$
are divisible 
by $p$ for any pairs $(m,n)$ with $m\equiv m_0 \pmod{N_q}, \ n\equiv n_0 \pmod{N_q}$ and $n>3$. Furthermore, we show that
if we assume that ERH holds, then there exists the pair $(m_0,n_0)$.
\end{abstract}

\section{Introduction}\label{sect:intro}
Let $N$ be a natural number.
Some infinite families of pairs of quadratic fields like $\Q(\sqrt{D})$ and $\Q (\sqrt{mD})$
with class numbers divisible by $N$ were given by Scholz \cite{S} ($N=3$), Komatsu \cite{K1,K2} ($N=3$, arbitrary $N$),
and Iizuka, Konomi and Nakano \cite{IKN} ($N=3,5,7)$.
In the previous paper \cite{AK2}, the authors constructed such an infinite family
in the case $N=5$ explicitly by using the Fibonacci numbers $F_n$.

\begin{theorem}[\cite{AK2}]\label{theo:JNT}
For $n\in {\mathcal N}:= \{n\in \N \,|\, n\equiv \pm 3 \pmod{500},\ n\not\equiv 0\pmod{3}\}$,
the class numbers of both $\mathbb Q(\sqrt{2-F_n})$ and $\mathbb Q(\sqrt{5(2-F_n)})$ are
divisible by $5$.  Moreover, the set of pairs 
$$\{(\mathbb Q(\sqrt{2-F_n}), \mathbb Q(\sqrt{5(2-F_n)}))\,|\,n\in {\mathcal N}\}$$
is infinite. 
\end{theorem}

The purpose of this paper is to give this type of an explicit infinite family of pairs of
imaginary cyclic fields of degree $(p-1)/2$ with both class numbers divisible by $p$
for any prime numbers $p$ such that $p\equiv 5\pmod{8}$.

Let $p$ be a prime number with $p\equiv 5 \mod{8}$
and let $\zeta:=\zeta_p$ be a primitive $p$th root of unity.
Let $\delta$ be a generator of $\Gal(\Q(\zeta)/\Q)$ and put
$\delta_0:=\delta^{(p-1)/4}$. Moreover, we put
$\omega_0:=\zeta+\zeta^{\delta_0}+\zeta^{\delta_0^2}+\zeta^{\delta_0^3}$.
Then $k_0:=\Q(\omega_0)$ is the unique subfield of $\Q(\zeta)$ of degree $(p-1)/4$.
Let $u_p>1$ be the fundamental unit of $k=\Q(\sqrt{p})$ and denote
\[
u_p=\frac{t+b\sqrt{p}}{2}\quad (t,b\in\Z,\ t,b>0).
\]
We use the following general linear recurrence sequences instead of Fibonacci numbers.
By using the trace $t$ of $u_p$, define two sequences $\{\F_n\}$, $\{\L_n\}$ by
\begin{equation}\label{eq:seqs}
\begin{cases}
\F_0:=0,\ \F_1:=1,\ \F_{n+2}:=t\F_{n+1}+\F_{n}\ (n\in \Z),\\
\L_0:=2,\ \L_1:=t,\ \L_{n+2}:=t\L_{n+1}+\L_{n}\ (n\in \Z).
\end{cases}
\end{equation}
For integers $m,n$ and a prime number $q\,(\ne p)$, we put
\begin{align*}
D_{m,n} &:={\mathcal L}_m(2{\mathcal F}_m-{\mathcal F}_n {\mathcal L}_m)b,\\
N_q&:=\begin{cases}
\lcm(p^2(p-1),q-1) & \displaystyle{\text{if}\ \left(\frac{p}{q}\right)=1},\\
\lcm(p^2(p-1),2(q+1)) & \displaystyle{\text{if}\ \left(\frac{p}{q}\right)=-1}.
\end{cases}
\end{align*}
When $m$ and $n$ are odd and $n>3$, $D_{m,n}$ is negative since
${\mathcal F}_{-m}=(-1)^{m+1}{\mathcal F}_{m}$ and ${\mathcal L}_{-m}=(-1)^m {\mathcal L}_{m}$.

In this paper, we first prove that if there exists a pair $(m_0,n_0)$ of integers and a prime
number $q$ satisfying certain congruence relations (Main Theorem~\ref{main:1} (i), (ii)), 
then the class numbers of both cyclic fields $k_0(\sqrt{D_{m,n}})$ and
$k_0(\sqrt{pD_{m,n}})$ of degree $(p-1)/2$ are divisible by $p$ for any pairs $(m,n)$
such that $m\equiv m_0 \pmod{N_q}, n\equiv n_0 \pmod{N_q}$ and $n>3$.
For the proof, we use the fundamental unit of $k=\Q(\sqrt{p})$, certain units which are
roots of a parametric quartic polynomial,  Kummer theory,  Gauss sums and 
Jacobi sums. Note that the fields $k_0(\sqrt{D_{m,n}})$ and $k_0(\sqrt{pD_{m,n}})$ are both
imaginary and their maximal real subfields are both $k_0=\Q(\omega_0)$.
It is expected that the class number of $k_0$ is not divisible by $p$
(Vandiver's conjecture).
Moreover, there are some examples in which 
the class numbers of both
$k_0(\sqrt{D_{m,n}})$ and $k_0(\sqrt{pD_{m,n}})$ are divisible by $p$, but that of neither
$\Q(\sqrt{D_{m,n}})$ nor $\Q(\sqrt{pD_{m,n}})$ is divisible by $p$
(see Remark~\ref{rem:JNT} (1) in \S\ref{sect:main}).
Next, we show that there exists the pair $(m_0,n_0)$ if we assume that ERH holds.
For the proof, we treat some curves on finite fields and use 
a consequence of Weil's theorem and a result of
Lenstra~\cite{L} which is a generalization
of Artin conjecture on primitive roots.

\section{Main Theorems}\label{sect:main}
Let $p$ be a prime number with $p\equiv 5 \pmod{8}$ and $\{ \F_n \}, \{\L_n \}$
be the recurrence sequences defined in \S\ref{sect:intro}. For integers $m,n$ and
a prime number $q\,(\ne p)$, we put 
\begin{align*}
\alpha=\alpha(m,n)&:=
\frac{\L_n\L_m+(\L_m\F_n-2\F_m)b\sqrt{p}}{2},\\
f_{\alpha}(X) &:=X^4-TX^3+({N}+2)X^2-TX+1,\\
f_{\alpha,q} (X)&:=f_{\alpha} \bmod{q}  \in \mathbb F_q [X],
\end{align*}
where $N:=N_{k/\Q}(\alpha)$, $T:=\Tr_{k/\Q}(\alpha)$. 
\begin{main}\label{main:1}
We assume that there exist integers 
$m_0, n_0$ with
$m_0 \equiv n_0 \equiv 1\pmod{2}$ and
a prime number $q$ such that 
\begin{itemize}
\item[$\rm{(i)}$] $({\mathcal L}_{m_0} {\mathcal F}_{n_0}-2{\mathcal F}_{m_0})b \equiv 0\mod{p^2}$,
\item[$\rm{(ii)}$]  $q \nmid 2bp$ and $f_{\alpha_0,q}(a)=0$ for some $i\in\{1,2,4\}$ and
$a\in \mathbb F_{q^i} \setminus \mathbb F_{q^i}^p$, where $\alpha_0:=\alpha(m_0,n_0)$.
\end{itemize}
Then for any pairs 
\[
(m,n)\in{\mathcal N}:=\{(m,n)\in\Z^2 \,|\, m\equiv m_0 \pmod{N_q},\
n\equiv n_0 \pmod{N_q},\ n>3\}, 
\]
the class numbers of both imaginary cyclic fields 
$k_0(\sqrt{D_{m,n}})$ and $k_0(\sqrt{pD_{m,n}})$ of degree $(p-1)/2$ are divisible by $p$.
Moreover, the set of pairs 
\[
\{(k_0(\sqrt{D_{m,n}}),k_0(\sqrt{pD_{m,n}}))  \,|\, (m,n) \in
{\mathcal N} \}
\]
 is infinite.
\end{main}
\begin{remark}\label{rem:JNT}
(1)
Let $p=13$. Then $t=3$, $b=1$, and
$(q,m_0,n_0)=(53,15,55)$ satisfies the conditions (i), (ii) of Main Theorem 1,
and hence the class numbers of both $k_0(\sqrt{D_{m_0,n_0}})$ and $k_0 (\sqrt{pD_{m_0,n_0}})$
are divisible by $p$.
In this case, the class numbers of $\Q(\sqrt{D_{m_0,n_0}})$ and $\Q(\sqrt{pD_{m_0,n_0}})$ are
\[7102491402551842304=2^9\cdot 7\cdot 1981721931515581\]
and
\[59331908185385308160=2^{12}\cdot 5\cdot 
2897065829364517,\]
respectively, and neither of them is divisible by $p=13$, where 
\begin{align*}
D_{m_0,n_0} &=-35297949870282964311195913270006746882588864\\
&=-2^6\cdot 3^2\cdot 13^2\cdot 61 \cdot 109 \cdot 131\cdot 211 \cdot 1063\cdot 2725164213221
\cdot 681089630669633.
\end{align*}
As for how to find $(q,m_0,n_0)$, see Example~\ref{ex:p=5} (2) in \S8. 

(2)
Main Theorem~\ref{main:1} implies the previous theorem (Theorem~\ref{theo:JNT} in
\S\ref{sect:intro}). For the details, see Example~\ref{ex:2} in \S8.
\end{remark}

\begin{main}\label{main:2}
Assume that ERH holds. Then there exist the integers $m_0,n_0$
and the prime number $q$ as in Main Theorem~\ref{main:1}.
\end{main}

\begin{remark}\label{rem:ERH}
``ERH'' means the extended Riemann hypothesis for $k(\zeta_n,\sqrt[n]{u_p})$
with every square free integers $n>0$.
\end{remark}

\section{The framework}

Let $p$ be a prime with $p\equiv 5 \mod{8}$ and put $k:=\Q(\sqrt{p})$.
Let $\alpha\in{\mathcal O}_k\setminus\Z$ with $\alpha^2-4\not\in\Z^2$.
Define the polynomial $f_{\alpha}(X)$ by
$$f_{\alpha}(X):=X^4-TX^3+({N}+2)X^2-TX+1,$$
where $N:=N_{k/\Q}(\alpha)$, $T:=\Tr_{k/\Q}(\alpha)$. 
From the assumptions $\alpha \in {\mathcal O}_k\setminus \Z$ and
$\alpha^2-4 \not\in \Z^2$, $f_{\alpha}(X)$ is irreducible over $\mathbb Q$
(cf.\ \cite[Proposition~2.1(1)]{AK}).
Let $L$ be the splitting field of $f_{\alpha}(X)$ over $\Q$.
We can easily verify that $T^2-4N>0$.
Hence if
\begin{equation}\label{A1}
\alpha^2-4>0\ \ \text{and}\ \ (N+4)^2-4T^2\in p\mathbb Q^2\tag{A1}
\end{equation}
hold, then $L$ is a real cyclic quartic field with $k\subset L$
(cf.\ \cite[Proposition~2.1 (2), Lemma~2.4]{AK}).
Moreover $L$ is not contained 
in $\Q(\zeta_p+\zeta_p^{-1})$ since $4\nmid [\Q(\zeta_p+\zeta_p^{-1}):\Q]=(p-1)/2$,
and hence $L \not\subset \Q(\zeta_p)$.
Put $\zeta:=\zeta_p$, $\omega:=\zeta+\zeta^{-1}$ and $\widetilde{L}:=L(\zeta)$.
Since $\Gal(\widetilde{L}/\Q)\simeq C_{p-1}\times C_2$,
$\widetilde{L}$ has two quadratic subfields other than $k$. We denote them by 
$K$ and $K'$.
Then we see that
$
\Gal(\widetilde{L}/K)\simeq \Gal(\widetilde{L}/K')\simeq C_{p-1}.
$
Let $\tau$ and $\tau'$ be a generator of $\Gal(\widetilde{L}/K)$ and
$\Gal(\widetilde{L}/K')$, respectively, 
whose restrictions to $\Q(\zeta)$ are the generator $\delta$ of $\Gal(\Q(\zeta)/\Q)$,
and put $\tau_0:=\tau^{\frac{p-1}{4}}$,
$\tau_0':={\tau'}^{\frac{p-1}{4}}$.
Then $\Q(\omega_0)$ is the unique subfield of $\Q(\zeta)$ of degree $(p-1)/4$, where
$$\omega_0:= \zeta+\zeta^{\tau_0}+\zeta^{\tau_0^2}+\zeta^{\tau_0^3}
= \zeta+\zeta^{\tau'_0}+\zeta^{{\tau'_0}^2}+\zeta^{{\tau'_0}^3}.$$
Since $\Gal(K(\omega)/\Q(\omega_0))\simeq C_2\times C_2$, $K(\omega)/\Q(\omega_0)$
has three proper subextensions $\mathbb Q(\omega)$, $K(\omega_0)$
and $K'(\omega_0)$. Put $K_0:=K(\omega_0)$ and $K'_0:=K'(\omega_0)$.
(See Figure 1.)

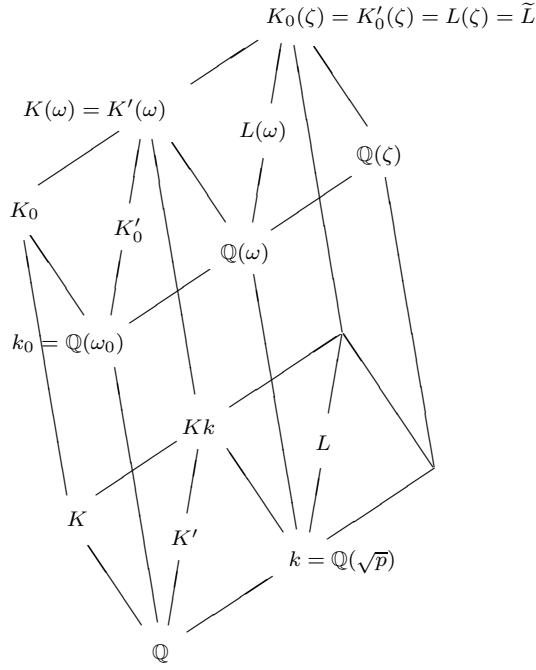
\begin{figure}[H]
\caption{A diagram of $\widetilde{L}/\Q$}
\label{figure:1}
\vspace{10mm}
\begin{center} 
{\small
\begin{picture}(81,60)
\put(-27,-43){$k_0=\Q(\omega_0)$}
\put(0,-30){\line(-2,3){20}}
\put(10,-30){\line(1,6){4}}
\put(17,-33){\line(3,2){30}}
\put(-28,7){$K_0$}
\put(11,0){$K'_0$}
\put(51,-9){$\Q(\omega)$}
\put(-17,19){\line(3,2){30}}
\put(17,12){\line(1,6){4}}
\put(53,5){\line(-2,3){20}}
\put(63,5){\line(1,6){4}}
\put(70,2){\line(3,2){30}}
\put(-23,45){$K(\omega)=K'(\omega)$}
\put(58,37){$L(\omega)$}
\put(102,28){$\Q(\zeta)$}
\put(40,57){\line(3,2){25}}
\put(70,50){\line(1,6){3.7}}
\put(106,41){\line(-2,3){20}}
\put(68,80){$K_0(\zeta)=K'_0(\zeta)=L(\zeta)=\widetilde{L}$}
\end{picture}

\begin{picture}(38,116)
\put(4,-43){$\Q$}
\put(0,-30){\line(-2,3){20}}
\put(10,-30){\line(1,6){4}}
\put(17,-33){\line(3,2){32}}
\put(-28,7){$K$}
\put(11,0){$K'$}
\put(55,-8){$k=\Q(\sqrt p)$}
\put(-17,19){\line(3,2){30}}
\put(17,12){\line(1,6){4}}
\put(53,5){\line(-2,3){21}}
\put(63,5){\line(1,6){4}}
\put(69.5,2.2){\line(3,2){40}}
\put(15,42){$Kk$}
\put(65,36){$L$}
\put(33,52){\line(3,2){41.4}}
\put(70,50){\line(1,6){4.9}}
\put(109,29.5){\line(-2,3){33.5}}
\put(20,57){\line(-1,6){16}}
\put(75.1,80.5){\line(-1,6){18}}
\put(109.5,29.6){\line(-1,6){18.2}}
\put(-28,19){\line(-1,6){16}}
\put(58,6){\line(-1,6){16}}
\put(6,-29){\line(-1,6){16}}
\end{picture}
}
\end{center} 
\vspace{19mm}
\end{figure}
In the following, we will construct an unramified cyclic extension of $K_0$
of degree $p$. (We can do the same argument when $K_0$ is replaced by $K_0'$.)
Let $\varepsilon,\varepsilon^{-1},\eta,\eta^{-1}$ be the roots of $f_{\alpha}(X)$ with
$\varepsilon+\varepsilon^{-1}=\alpha,\ \eta+\eta^{-1}=\overline{\alpha}$
(cf.\ \cite[Lemmas~2.2, 2.3]{AK}).
Then we may assume that
\begin{align*}
\tau&:\varepsilon \mapsto \eta \mapsto \varepsilon^{-1} \mapsto \eta^{-1}, \\
\tau'&:\varepsilon \mapsto \eta^{-1} \mapsto \varepsilon^{-1} \mapsto \eta
\end{align*}
(cf.\ \cite[Lemma~1]{AK2}). Since $(p-1)/4$ is odd, we may assume
\begin{align*}
\tau_0&:\varepsilon \mapsto \eta \mapsto \varepsilon^{-1}\mapsto \eta^{-1}, \\
\tau_0'&:\varepsilon \mapsto \eta^{-1} \mapsto \varepsilon^{-1}\mapsto \eta.
\end{align*}
Here we may assume that
\begin{equation}\label{eq:iota}
\zeta^{\tau}=\zeta^{\iota}, \qquad \zeta^{\tau'}=\zeta^{\iota},
\end{equation}
where $\iota$ is a primitive root modulo $p$.
Setting $\iota_0:=\iota^{\frac{p-1}{4}}$,
we have
$$\Gal(K_0(\zeta)/K_0)=\langle\tau_0\rangle,\qquad \zeta^{\tau_0}=\zeta^{\iota_0}.$$
We define an element $t(K_0)\in\Z[\Gal(K_0(\zeta)/K_0)]$ by
\[
t(K_0):=\iota_0^3+\tau_0\iota_0^2+\tau_0^2\iota_0+\tau_0^3\in\Z[\Gal(K_0(\zeta)/K_0)],
\]
and a subset $T(K_0)$ of $\Z[\Gal(K_0(\zeta)/K_0)]$ by
\[
T(K_0):=\{t'(K_0)\in\Z[\Gal(K_0(\zeta)/K_0)]\,|\,{}^{\exists} n \in (\Z/p\Z)^{\times}\
{\rm s.t.}\ t'(K_0)\equiv nt(K_0) \mod{p}\}.
\]
Moreover, we define a subset ${\calM}_{\tau}$ of $\widetilde{L}^{\times}$ by
\[
{\calM}_{\tau}:=\{\gamma \in \widetilde{L}^{\times}\,|\, \gamma^{t(K_0)}\not\in\widetilde{L}^p\}.
\]

\begin{proposition}\label{prop:4p-cyclic}
For any $\gamma \in {\calM}_{\tau}$ and $t'(K_0)\in T(K_0)$, 
$\widetilde{L}(\sqrt[p]{\gamma^{t'(K_0)}})/K_0$ is a cyclic extension of degree $4p$.
\end{proposition}

\begin{proof}
From a direct calculation, we have
\[
(\tau_0-\iota_0)t(K_0)\equiv \tau_0t(K_0)-\iota_0t(K_0) =1-\iota_0^4 \equiv 0\mod{p}.
\]
This implies $\gamma^{t'(K_0)(\tau_0-\iota_0)} \in \widetilde{L}^p$.
By \cite[Proposition 1.1]{IK}, therefore, $\widetilde{L}(\sqrt[p]{\gamma^{t'(K_0)}})/K_0$
is a cyclic extension of degree $4p$.
\qed
\end{proof}

\begin{remark}\label{rem:2}
Let $\gamma \in {\calM}_{\tau}$. Then it follows from the definition of $T(K_0)$ that
\[\widetilde{L}(\sqrt[p]{\gamma^{t(K_0)}})=\widetilde{L}(\sqrt[p]{\gamma^{t'(K_0)}})\]
for any $t'(K_0)\in T(K_0)$.
\end{remark}

Now assume
\begin{equation}\label{A2}
\varepsilon\in {\calM}_{\tau}\tag{A2}
\end{equation}
and put $\beta:=\sqrt[p]{\varepsilon^{t(K_0)}}$.
Then by Proposition~\ref{prop:4p-cyclic}, $\widetilde{L}(\beta)/K_0$ is
a cyclic extension of degree $4p$.
Let $E$ be the unique subextension of $\widetilde{L}(\beta)/K_0$
such that $E/K_0$ is a cyclic extension of degree $p$. (See Figure~\ref{figure:2}.)
Since $\varepsilon$ is a unit, we see by Kummer theory that
\begin{align*}
\text{$E/K_0$ is unramified}\ &\Longleftrightarrow\ 
\text{$\widetilde{L}(\beta)/\widetilde{L}$ is unramified}\\
&\Longleftrightarrow\
{}^{\exists} x\in \widetilde{L}^{\times}\ {\rm s.t.}\ x^p \equiv \varepsilon^{t(K_0)}
\mod{p(\zeta_p -1) {\mathcal O}_{\widetilde{L}}}
\end{align*}
(cf.\ \cite[Exercise\ 9.3 (b)]{W}).
Thus, under the assumption
\begin{equation}\label{A3}
{}^{\exists} x\in \widetilde{L}^{\times}\ {\rm s.t.} \ x^p \equiv \varepsilon^{t(K_0)}
\mod{p(\zeta_p -1) {\mathcal O}_{\widetilde{L}} },\tag{A3}
\end{equation}
$E/K_0$ is an unramified cyclic extension of degree $p$, and hence
the class number of $K_0$ is divisible by $p$.

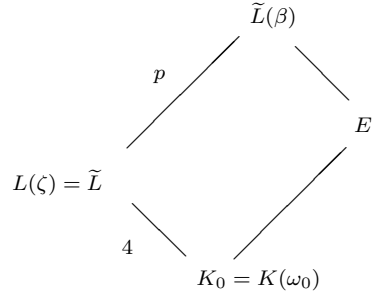
\begin{figure}[H]
\caption{A diagram of $\widetilde{L}(\beta)/K_0$}
\label{figure:2}
\vspace{2mm}
\begin{center} 
{\small
\begin{picture}(0,45)
\put(-12,-60){$K_0=K(\omega_0)$}
\put(-40,-48){$4$}
\put(-28,18){$p$}
\put(-16,-49){\line(-1,1){20}}
\put(2,-50){\line(1,1){42}}
\put(-81,-24){$L(\zeta)=\widetilde{L}$}
\put(47,-2){$E$}
\put(-38,-8){\line(1,1){42}}
\put(45,10){\line(-1,1){20}}
\put(8,39){$\widetilde{L}(\beta)$}
\end{picture}
}\vspace{23mm}
\end{center} 
\end{figure}

\begin{lemma}\label{lem:A3}
Assume that 
$(N+4)^2-4T^2 \equiv 0\pmod{p^5}$.
Then there exists $x\in \widetilde{L}^{\times}$ such that
\[x^p \equiv \varepsilon^{t(K_0)}
\mod{p(\zeta_p -1) {\mathcal O}_{\widetilde{L}} },\]
that is, $\mathrm{(\ref{A3})}$ holds.
\end{lemma}
\begin{proof}
By $(N+4)^2-4T^2 \equiv 0\pmod{p^5}$ and 
$p{\mathcal O}_{\widetilde{L}} =(\zeta_p-1)^{p-1}{\mathcal O}_{\widetilde{L}}$,
we have
\begin{equation}\notag
(\alpha^2-4)(\overline{\alpha}^2-4)=(N+4)^2-4T^2\equiv 0
\mod{(\zeta_p-1)^{5(p-1)}{\mathcal O}_{\widetilde{L}}}.
\end{equation}
Hence we have
\begin{equation}\notag
\alpha^2-4 \equiv 0 \mod{(\zeta_p-1)^{5(p-1)/2}{\mathcal O}_{\widetilde{L}}}
\end{equation}
or
\begin{equation}\notag
\overline{\alpha}^2-4 \equiv 0 \mod{(\zeta_p-1)^{5(p-1)/2}{\mathcal O}_{\widetilde{L}}}.
\end{equation}
Since the ideal $(\zeta_p-1)^{5(p-1)/2}  {\mathcal O}_{\widetilde{L}}$ is invariant
under the action of $\Gal(\widetilde{L}/\Q)$,
we have 
\begin{equation}\notag
\alpha^2-4 \equiv \overline{\alpha}^2-4 \equiv 0
\mod{ (\zeta_p-1)^{5(p-1)/2} {\mathcal O}_{\widetilde{L}}},
\end{equation}
and hence,
\begin{equation}\label{eq:epsilon}
\varepsilon=\frac{\alpha+\sqrt{\alpha^2-4}}{2} \equiv \frac{\alpha}{2}
\mod{ (\zeta_p-1)^{5(p-1)/4} {\mathcal O}_{\widetilde{L}}}.
\end{equation}
By $\tau_0=\tau^{\frac{p-1}{4}}$, therefore, we have 
\begin{equation}\label{eq:epsilon2}
\varepsilon^{\tau_0} \equiv \frac{\overline{\alpha}}{2}
\mod{ (\zeta_p-1)^{5(p-1)/4} {\mathcal O}_{\widetilde{L}}}.
\end{equation}
Now we have
$\iota_0^2=\iota^{\frac{p-1}{2}}\equiv -1\mod{p}$.
Let us express $\iota_0^2=ps-1$ for some $s\in\Z$.
Then by (\ref{eq:epsilon}) and (\ref{eq:epsilon2}), we have
\begin{align*} 
\varepsilon^{t(K_0)}
&\equiv \left( \frac{\alpha}{2} \right)^{\iota_0^3+\tau_0\iota_0^2+\tau_0^2\iota_0+\tau_0^3}\\
&\equiv \left( \frac{\alpha}{2} \right)^{\iota_0(ps-1)} 
\left( \frac{\overline{\alpha}}{2} \right)^{ps-1}
\left( \frac{\alpha}{2} \right)^{\iota_0}\cdot\frac{\overline{\alpha}}{2}\\
&=\left\{ \left(\frac{\alpha}{2} \right)^{\iota_0}\cdot\frac{\overline{\alpha}}{2}\right\}^{ps} 
\mod{(\zeta_p-1)^{5(p-1)/4} {\mathcal O}_{\widetilde{L}}} .
\end{align*}
Hence by
\[
(\zeta_p-1)^{\frac{5(p-1)}{4}} {\mathcal O}_{\widetilde{L}}
= p(\zeta_p-1)^{\frac{p-1}{4}}{\mathcal O}_{\widetilde{L}}
\subset p(\zeta_p-1){\mathcal O}_{\widetilde{L}},
\]
we get the assertion.
\qed
\end{proof}

In \S\ref{2-conditions}, we will show that $\alpha=\alpha(m,n)$
with $(m,n)\in{\mathcal N}$, which is defined in \S2, satisfies conditions
(\ref{A1}), (\ref{A2}) and (\ref{A3}).

\section{The fundamental unit of $\Q(\sqrt{p})$ and Lucas sequences}\label{sect:F-L}

In this section, let $p$ be a prime with $p\equiv 1 \mod{4}$.
Then the norm of the fundamental unit
$$u_p=\frac{t+b\sqrt{p}}{2}\quad (t,b\in\Z,\ t,b>0)$$
of $\mathbb Q(\sqrt{p})$ is equal to $-1$
(see, for example, \cite[p.279, Theorem~11.5.4]{AW}, \cite[p.316, Exercise~5]{T}).
By using the trace $t$ of $u_p$, we define
two sequences $\{ {\calF}_n \}, \{ {\calL}_n \}$ by (\ref{eq:seqs}).
The sequences  $\{ {\calF}_n \}$ and $\{ {\calL}_n \}$ are called the Lucas sequence
and the companion  Lucas sequence, respectively, associated to
the characteristic polynomial $P(X)=X^2-tX-1$,
 which are known to satisfy the following properties: 
\begin{align}
&\calF_n=\frac{u_p^n-\overline{u}_p^n}{u_p-\overline{u}_p},\quad\calL_n=u_p^n+\overline{u}_p^n,\\
&\calL_ n^2-b^2p{\calF}_n^2=(-1)^n4,\label{eq:normunit}\\
&\F_{n+m}=\F_{n}\F_{m+1}+\F_{n-1}\F_{m},\label{eq:FLLFF}\\
&\L_{n+m}-(-1)^m\L_{n-m}=b^2p\F_{n}\F_{m}\label{eq:LLFF},
\end{align}
where $\overline{u}_p$ denotes the Galois conjugate of $u_p$ (see, for example, \cite[Chap.~2, IV]{R}).
\begin{lemma}\label{cor1}
For any integer $n\in\Z$, we have the following:

$(1)$ $\F_{2n+1}=\F_{n+1}^2+\F_{n}^2$.

$(2)$ $\F_{n}^2-\F_{n+1}^2=(-t\L_{2n+1}-4(-1)^n)/b^2p$.

$(3)$ $\F_{n}\F_{n+1}=(\L_{2n+1}-(-1)^nt)/b^2p$.
\end{lemma}
\begin{proof}
(1) The assertion follows from (\ref{eq:FLLFF}) immediately.

(2) From (\ref{eq:LLFF}), we get
\begin{align*}
\L_{2n}-(-1)^n\L_0&=b^2p\F_n^2,\\
\L_{2n+2}-(-1)^{n+1}\L_0&=b^2p\F_{n+1}^2,
\end{align*}
and so
$$\L_{2n}-\L_{2n+2}-\{(-1)^n-(-1)^{n+1}\}\L_0=b^2p(\F_n^2-\F_{n+1}^2).$$
Since $\L_0=2$ and $\L_{2n+2}=t\L_{2n+1}+\L_{2n}$, we obtain
$$-t\L_{2n+1}-4(-1)^n=b^2p(\F_n^2-\F_{n+1}^2).$$

(3) From (\ref{eq:LLFF}) and $\L_1=t$, we get
\[
\L_{2n+1}-(-1)^{n}t=b^2p\F_{n+1}\F_n
\]
as desired.
\qed
\end{proof}
\begin{lemma}\label{lem:modp^2}
The period of $\{ \F_n \} \bmod{p^2}$ $($resp. $ \{ \L_n \} \bmod{p^2})$ divides $p^2(p-1)$  $($resp. $p(p-1))$.
\end{lemma}
\begin{proof}
For any integer $n\geq 4$, we have
\begin{align*}
u_p^n &=2^{-n}(t+b\sqrt{p})^n\\
& \equiv 2^{-n} \left( t^n +\binom{n}{1} t^{n-1} b\sqrt{p}+\binom{n}{2} t^{n-2}b^2 p
+\binom{n}{3} t^{n-3}b^3p\sqrt{p}+\binom{n}{4}t^{n-4}b^4p^2 \right),\\
\overline{u}_p^n &=2^{-n}(t-b\sqrt{p})^n\\
& \equiv 2^{-n} \left( t^n -\binom{n}{1} t^{n-1} b\sqrt{p}+\binom{n}{2} t^{n-2}b^2 p
-\binom{n}{3} t^{n-3}b^3p\sqrt{p}+\binom{n}{4}t^{n-4}b^4p^2 \right)\\
& \pmod{bp^2\sqrt{p} {\mathcal O}_k},
\end{align*}
and hence
\begin{align*}
u_p^n-\overline{u}_p^n &\equiv 2^{-n+1} \left(
\binom{n}{1} t^{n-1} b\sqrt{p}+\binom{n}{3} t^{n-3}b^3p\sqrt{p} \right) \pmod{bp^2\sqrt{p} {\mathcal O}_k}.
\end{align*}
Therefore, we get
\[ 
\F_n =\frac{u_p^n-\overline{u}_p^n}{u_p-\overline{u}_p} \equiv 2^{-n+1} 
\left( \binom{n}{1} t^{n-1} +\binom{n}{3} t^{n-3}b^2p \right)
\pmod{p^2}.
\]
Assume that integers $m,n$ satisfy $m\equiv n \pmod{p^2 (p-1)}$.
Then we have 
\[
\binom{m}{1} \equiv \binom{n}{1}, \binom{m}{3} \equiv \binom{n}{3} 
\pmod{p^2} 
\]
 and $2^m\equiv 2^n, t^m\equiv t^n \pmod{p^2}$ since 
$2$ and $t$ are two invertible elements of $\Z/p^2\Z$ and 
the order of the
cyclic group $(\Z/p^2\Z)^{\times}$ is $p(p-1)$.
It concludes that $\F_m \equiv \F_n  \pmod{p^2}$, and the period of $\{ \F_m \} \bmod{p^2}$
divides $p^2(p-1)$.
Similarly, by using
\[
\L_n =u_p^n+\overline{u}_p^n \equiv 2^{-n+1} 
\left( t^n +\binom{n}{2} t^{n-2}b^2p \right)
\pmod{p^2},
\]
and 
\[
\binom{n}{2} \equiv \binom{m}{2} \pmod{p}
\]
for integers $m,n$ satisfying $m\equiv n \pmod{p(p-1)}$,
we see that the period of $\{ \L_m \} \bmod{p^2}$ divides $p(p-1)$.

\qed
\end{proof}

Now we define two integers $A,B\in\Z$ by
\begin{equation}\label{eq:pAB}
p=A^2+B^2,\ A\equiv -1\mod{4}.
\end{equation}
The sign of $B$ will be determined after the following lemma.

\begin{lemma}\label{lem:div}
Under the above notation, either
$p\mid At+ 2B$ or $p\mid At- 2B$ but not both holds.
\end{lemma}
\begin{proof}
Since $p=A^2+B^2$ and $t^2+4=b^2p$, we have
\begin{align}
(At+ 2B)(At- 2B)&=A^2t^2- 4B^2=(p-B^2)(b^2p-4)-4B^2\label{eq:At2B}=p(b^2p-4-b^2B^2).
\end{align}
Assume that both $p\mid At+ 2B$ and $p\mid At- 2B$ hold.
Then $4B$ is divisible by $p$, and so is $B$.
This leads a contradiction.
\qed
\end{proof}
Determine the even integer $B$ to satisfy conditions (\ref{eq:pAB}) and $p\mid At+ 2B$.

\begin{lemma}\label{lem:2}
Under the above notation, we have
\begin{equation}\label{eq:lem2}
bp\geq|At\pm 2B|.
\end{equation}
\end{lemma}

\begin{proof}
We see
\begin{align*}
(bp)^2-(At\pm 2B)^2&=b^2p\cdot p-(A^2t^2\pm 4ABt+4B^2)\\
&=(t^2+4)(A^2+B^2)-(A^2t^2\pm 4ABt+4B^2)\\
&=B^2t^2+4A^2\mp 4ABt\\
&=(Bt\mp 2A)^2\geq 0.
\end{align*}
From this together with $bp>0$, we obtain (\ref{eq:lem2}).
\qed
\end{proof}
\begin{lemma}\label{lem:4}
Under the above notation, we have
\begin{equation*}
\sqrt{\frac{bp-(At+2B)}{2p}}\sqrt{\frac{bp+(At+2B)}{2p}}=\frac{|Bt-2A|}{2p}.
\end{equation*}
\end{lemma}
\begin{proof}
As we have seen in the proof of Lemma~\ref{lem:2}, the equation
$$(bp)^2-(At+ 2B)^2=(Bt-2A)^2$$
holds. Then we get the assertion.
\qed
\end{proof}

\begin{definition}
Define two real numbers $x_0,y_0\in\R$ by 
\begin{align*}
x_0:=\sqrt{\frac{bp+(At+2B)}{2p}}\ \ \text{and}\ \
y_0:=\kappa\sqrt{\frac{bp-(At+2B)}{2p}}.
\end{align*}
Here, $\kappa$ is equal to $1$ or $-1$ which satisfies
$$x_0y_0=\frac{Bt-2A}{2p}.$$
Namely, $y_0$ and $Bt-2A$ are the same signs.
\end{definition}

\begin{lemma}\label{lem:integer}
Under the above notation (especially, we assume that  $p \mid At+2B$), 
we have  $x_0,y_0\in \Z$.
\end{lemma}
\begin{proof}
Since
\begin{equation}\label{eq:1}
x_0^2y_0^2=\frac{bp+(At+2B)}{2p}\cdot\frac{bp-(At+2B)}{2p} =\left( \frac{Bt-2A}{2p} \right)^2,
\end{equation}
it is sufficient to show that $x_0^2= ( bp+(At+2B))/2p$ and $y_0^2=( bp-(At+2B))/2p$
are coprime. Assume, on the contrary, that $(x_0^2,y_0^2)\not=1$.
Then there exists a prime $\ell$ such that
\begin{equation}\label{eq:2}
\ell\mid x_0^2\ \ \text{and}\ \ \ell\mid y_0^2
\end{equation}
Then we have
\begin{equation}\label{eq:3}
\ell \mid x_0^2+y_0^2=b \ \ \text{and} \ \ \ell \mid x_0^2-y_0^2=\frac{At+2B}{p}.
\end{equation}
From these and (\ref{eq:At2B}), we have
$$
0 \equiv -4p \mod{\ell},
$$
and so ether $\ell=2$ or $\ell=p$ holds.
In the case where $\ell=p$, we have $p^2 \nmid p(b^2p-4-b^2B^2)$ because of $p\mid b$.
Then by (\ref{eq:At2B}), we have $p^2 \nmid At+2B$. This implies $p\nmid x_0^2=(At+2B)/2p$,
which contradicts to (\ref{eq:2}).
In the case where $\ell =2$, (\ref{eq:3}) and $2\nmid A$ imply both $2\mid b$ and $2\mid t$.
Moreover, $2\nmid A$ and $2\mid B$ imply $2^2\nmid Bt-2A$. Thus $(Bt-2A)/2p$ is odd.
Then by (\ref{eq:1}), we have 
$2\nmid x_0^2y_0^2$ 
which contradicts (\ref{eq:2}).
Therefore, $x_0^2= ( bp+(A+2B))/2p$ and $y_0^2=( bp-(At+2B))/2p$ are coprime.
The proof is complete.
\qed
\end{proof}

\begin{lemma}\label{lem:last}
Under the above notation, we have
\begin{equation}\label{eq:prop}
bp\F_{4n\pm1}\mp\L_{4n\pm1}A-2B=2p(x_0\F_{2n}\pm y_0\F_{2n\pm1})^2.
\end{equation}
\end{lemma}

\begin{proof}
By Lemma~\ref{cor1}, we have
\begin{align*}
2p(x_0\F_{2n}\pm y_0\F_{2n\pm1})^2
& =2p(x_0^2\F_{2n}^2+ y_0^2\F_{2n\pm1}^2\pm 2x_0y_0\F_{2n}\F_{2n\pm1})\\
& =2p\left(\frac{bp+(At+2B)}{2p}\F_{2n}^2+\frac{bp-(At+2B)}{2p}\F_{2n\pm1}^2\pm
\frac{Bt-2A}{p}\F_{2n}\F_{2n\pm1}\right)\\
& =bp(\F_{2n}^2+\F_{2n\pm1}^2)+(At+2B)(\F_{2n}^2-\F_{2n\pm1}^2)
\pm 2(Bt-2A)\F_{2n}\F_{2n\pm1}\\
& =bp\F_{4n\pm 1}+(At+2B)\dfrac{\mp t\L_{4n\pm 1}-4}{b^2p}
\pm 2(Bt-2A)\dfrac{\L_{4n\pm1}\mp t}{b^2p}\\
& =bp\F_{4n\pm 1}+\frac{1}{b^2p}(t^2+4)(\mp L_{4n\pm 1}A-2B)\\
& =bp\F_{4n\pm 1}\mp L_{4n\pm 1}A-2B.
\end{align*}
The proof is complete.
\qed
\end{proof}

\section{Quadratic subfields}

The aim of this section is to determine the quadratic subfields of $\widetilde{L}$.
Let the notations be as in \S3. 
For simplicity, we assume that the primitive root $\iota$ modulo $p$
defined in (\ref{eq:iota}) satisfies
\begin{equation}\label{eq:6.1}
t\equiv -2\iota^{\frac{p-1}{4}}\mod{p}.
\end{equation}
Indeed, for any primitive roots $\iota$ modulo $p$, we have
$\iota^{\frac{p-1}{2}}\equiv -1\mod{p}$, which implies that
either $t\equiv 2\iota^{\frac{p-1}{4}}\mod{p}$ or
$t\equiv -2\iota^{\frac{p-1}{4}}\mod{p}$ holds by using $t^2\equiv -4\mod{p}$.
If $t\equiv 2\iota^{\frac{p-1}{4}}\mod{p}$, then we replace $\iota$ by $-\iota$
which is also a primitive root modulo $p$.

We recall that the actions of $\tau$ and $\tau'$ on $\varepsilon,
\varepsilon^{-1},\eta$ and $\eta^{-1}$ are as follows:
\begin{align*}
\tau&:\varepsilon \mapsto \eta \mapsto \varepsilon^{-1} \mapsto \eta^{-1}, \\
\tau'&:\varepsilon \mapsto \eta^{-1} \mapsto \varepsilon^{-1} \mapsto \eta.
\end{align*}
Here we put
$$
S_0:=\sum_{k=0 \atop k\equiv 0 \mod{4}}^{p-2} \zeta^{\iota^k},\
S_1:=\sum_{k=0 \atop k\equiv 1 \mod{4}}^{p-2} \zeta^{\iota^k},\
S_2:=\sum_{k=0 \atop k\equiv 2 \mod{4}}^{p-2} \zeta^{\iota^k},\
S_3:=\sum_{k=0 \atop k\equiv 3 \mod{4}}^{p-2} \zeta^{\iota^k}.
$$
Then we can verify that
\begin{equation}\label{eq:S0-S3}
\tau,\tau':S_0 \mapsto S_1 \mapsto S_2 \mapsto S_3 \mapsto S_0.
\end{equation}
Moreover we
define the elements $\lambda,\mu\in L(\zeta)$ by
\begin{align*}
\lambda&:=(\varepsilon-\varepsilon^{-1})(S_0-S_2)+ (\eta-\eta^{-1})(S_1-S_3),\\
\mu&:=(\varepsilon-\varepsilon^{-1})(S_1-S_3)+ (\eta-\eta^{-1})(S_0-S_2).
\end{align*}

\begin{lemma}\label{lem:lminK}
We have $\lambda\in K,\ \mu\in K'$.
\end{lemma}
\begin{proof}
By (\ref{eq:S0-S3}), we can verify $\lambda^\tau=\lambda$ and $\mu^{\tau'}=\mu$.
So the assertion follows.
\qed
\end{proof}

In the following, we will compute $\lambda$ and $\mu$.
Let $\chi_{-p}$ be a character modulo $p$ of order $4$ with $\chi_{-p}(\iota)=i$.
Now we consider the Gauss sums $G(\chi_{-p})$ and $G(\overline{ \chi_{-p}})$ of
$\chi_{-p}$ and $\overline{ \chi_{-p}}=\chi_{-p}^3$, respectively.
Then we have
\begin{align*}
G(\chi_{-p})&:=\sum_{a\in(\Z/p\Z)^{\times}}\chi_{-p}(a)\zeta^a=(S_0-S_2)+i(S_1-S_3),\\
G(\overline{ \chi_{-p}})&:=\sum_{a\in(\Z/p\Z)^{\times}} 
\overline{\chi_{-p}}(a)\zeta^a=(S_0-S_2)-i(S_1-S_3),
\end{align*}
and hence,
\begin{align}
S_0-S_2&=\frac{1}{2} (G(\chi_{-p})+G(\overline{\chi_{-p}})),\label{gauss1}\\
S_1-S_3&=\frac{1}{2i} (G(\chi_{-p})-G(\overline{\chi_{-p}})).\label{gauss2}
\end{align}
Moreover we see from $p\equiv 5\mod{8}$ that
\begin{equation}\label{eq:GGp}
G(\chi_{-p})G(\overline{\chi_{-p}})=\chi_{-p}(-1)p=-p,
\end{equation}
(\cite[Theorem~1.1.4 (a)]{BE}).
Let 
$$J(\chi_{-p},\chi_{-p}):=\sum_{a\in \Z/p\Z} \chi_{-p} (a) \chi_{-p}(1-a)$$ 
be the Jacobi sum of $\chi_{-p}$.
Then we can write
\begin{equation}\label{eq:jacobi}
J(\chi_{-p},\chi_{-p})=c_4+id_4,
\end{equation}
where $c_4$ and $d_4$ are rational integers such that $c^2_4+d^2_4=p$,
$c_4\equiv -1 \mod{4}$ and $d_4\equiv c_4\iota^{\frac{p-1}{4}}$
(\cite[Theorems~3.2.1, 3.2.2, Table 3.2.1]{BE}).
By using notation in (\ref{eq:pAB}), we have $c_4=A$.
Moreover, it follows from the definition of $B$ that
$-At\equiv 2B \mod{p}$. From this together with $d_4\equiv c_4\iota^{\frac{p-1}{4}}$
and (\ref{eq:6.1}), we have $d_4=B$.

On the other hand, let $\chi_p$ be the character modulo $p$ of order $2$,
namely, $\chi_p(a)=(\frac{a}{p} )$ for any $a\in (\Z/p\Z)^{\times}$.
Noting that $p\equiv 5\mod{8}$ and $\chi_{-p}^2=\chi_p$, we have
\begin{align*}
G(\chi_p)&=\sqrt{p},\\
J(\chi_{-p},\chi_{-p})&=\frac{G(\chi_{-p})^2}{G(\chi^2_{-p})}
=\frac{G(\chi_{-p})^2}{G(\chi_{p})},
\end{align*}
(\cite[Theorem~1.2.4, 2.1.3]{BE}).
From these relations together with (\ref{eq:jacobi}), we obtain 
\begin{equation}\label{7.6}
G(\chi_{-p})^2=G(\chi_{p})J(\chi_{-p},\chi_{-p})=\sqrt{p}(A+iB),
\end{equation}
and hence by (\ref{eq:GGp}),
\begin{equation}\label{7.7}
G(\overline{\chi_{-p}})^2
=\frac{G(\chi_{-p})^2G(\overline{\chi_{-p}})^2}{G(\chi_{-p})^2}
=\frac{p^2}{\sqrt{p}(A+iB)}=\sqrt{p}(A-iB). 
\end{equation}
Thus it follows from (\ref{eq:GGp}), (\ref{7.6}) and (\ref{7.7}) that
\begin{align}
(G(\chi_{-p})+G(\overline{\chi_{-p}}))^2&=2\sqrt{p} A-2p,\label{7.8}\\
(G(\chi_{-p})-G(\overline{\chi_{-p}}))^2&=2\sqrt{p} A+2p.\label{7.9}
\end{align}

\begin{lemma}\label{lem:abc}
The following hold:

$(1)$ $(\varepsilon -\varepsilon^{-1})^2(S_0-S_2)^2 +(\eta-\eta^{-1})^2(S_1-S_3)^2=
-\dfrac{1}{2} {\rm Tr}_{k/\mathbb Q} \{ (\alpha^2-4)(p-\sqrt{p} A) \}$.

$(2)$  $ (\varepsilon -\varepsilon^{-1})^2(S_1-S_3)^2 +(\eta-\eta^{-1})^2(S_0-S_2)^2=
-\dfrac{1}{2} {\rm Tr}_{k/\mathbb Q} \{ (\alpha^2-4)(p+\sqrt{p} A) \}$.

$(3)$ $(\varepsilon -\varepsilon^{-1})(\eta-\eta^{-1})(S_0-S_2)(S_1-S_3)=
\dfrac{ \sqrt{p} B}{2} \sqrt{(\alpha^2-4)(\overline{\alpha}^2-4)}$.
\end{lemma}

\begin{proof}
(1) Recall $\varepsilon+\varepsilon^{-1}=\alpha$, $\eta+\eta^{-1}=\overline{\alpha}$. Then we have
\begin{equation}\label{eq:11}
(\varepsilon-\varepsilon^{-1})^2=\alpha^2-4,\ (\eta-\eta^{-1})^2=\overline{\alpha}^2-4,
\end{equation}
and hence by (\ref{gauss1}), (\ref{gauss2}), (\ref{7.8}) and (\ref{7.9}),
\begin{align*}
&(\varepsilon -\varepsilon^{-1})^2(S_0-S_2)^2 +(\eta-\eta^{-1})^2(S_1-S_3)^2\\
&\ =\frac{1}{4}\{(\alpha^2-4)(G(\chi_{-p})+G(\overline{\chi_{-p}}))^2
-(\overline{\alpha}^2-4)(G(\chi_{-p})-G(\overline{\chi_{-p}}))^2\}\\
&\ =\frac{1}{4}\{(\alpha^2-4)(2\sqrt{p}A-2p)-(\overline{\alpha}^2-4)(2\sqrt{p}A+2p)\}\\
&\ =-\frac{1}{2}\{(\alpha^2-4)(p-\sqrt{p}A)+(\overline{\alpha}^2-4)(p+\sqrt{p}A)\}\\
&\ =-\frac{1}{2} {\rm Tr}_{k/\mathbb Q} \{ (\alpha^2-4)(p-\sqrt{p} A) \}.
\end{align*}

(2) The assertion follows from a similar calculation to that of (1)

(3) Since $(\varepsilon-\varepsilon^{-1})(\eta-\eta^{-1} ) >0$ (\cite[Lemma 2]{AK2}),
it follows from (\ref{eq:11}) that
\[
(\varepsilon-\varepsilon^{-1})(\eta-\eta^{-1})=\sqrt{(\alpha^2-4)(\overline{\alpha}^2-4)}.
\]
Then by (\ref{gauss1}), (\ref{gauss2}), (\ref{7.6}) and (\ref{7.7}), we have
\begin{align*}
 (\varepsilon -\varepsilon^{-1})(\eta-\eta^{-1})(S_0-S_2)(S_1-S_3)
&=\sqrt{ (\alpha^2-4)(\overline{\alpha}^2-4)} \cdot 
\frac{1}{4i} (G(\chi_{-p})^2-G(\overline{\chi_{-p}})^2)\\
&=\frac{1}{4i}  \sqrt{ (\alpha^2-4)(\overline{\alpha}^2-4)} \cdot 2\sqrt{p} Bi\\
&=\frac{ \sqrt{p} B}{2} \sqrt{(\alpha^2-4)(\overline{\alpha}^2-4)},
\end{align*}
as desired.
\qed
\end{proof}

From now on, let the situation be as in our main theorems. Namely,
we define an element $\alpha\in k$ by 
\[
\alpha =\alpha (m,n):=\frac{{\calL}_n {\calL}_m +({\calL}_m{\calF}_n-2{\calF}_m)b\sqrt{p}}{2}.
\]
for $m,n\in\Z$.
Then we have the following lemma.

\begin{lemma}\label{lem:d2}
Assume that both $m$ and $n$ are odd. Then we have
\[
(N+4)^2-4T^2=\calL_m^2b^2p(\calL_m\calF_n-2\calF_m)^2.
\]
Especially, $(N+4)^2-4T^2\in p\Q^2$.
\end{lemma}

\begin{proof}
It follows from (\ref{eq:normunit}) that
\begin{align}
N &=\frac{{\calL}_n^2 {\calL}_m^2-({\calL}_m{\calF}_n-2{\calF}_m)^2b^2p}{4}
=\frac{(b^2p\calF_n^2-4)\calL_m^2-(\calL_m\calF_n-2\calF_m)^2b^2p}{4} \label{eq:N}\\
&=-{\calL}_m^2+{\calL}_m{\calF}_n{\calF}_mb^2p-{\calF}_m^2b^2p
=-({\calF}_m^2b^2p-4)+{\calL}_m{\calF}_n{\calF}_mb^2p-{\calF}_m^2b^2p\notag\\
&={\calF}_mb^2p({\calL}_m{\calF}_n-2{\calF}_m)+4.\notag
\end{align}
Hence by using $T^2-({\calL}_m{\calF}_n-2{\calF}_m)^2b^2p=4N$ and
(\ref{eq:normunit}), we have
\begin{align*}
(N+4)^2-4T^2 & =(N+4)^2-4\{({\calL}_m{\calF}_n-2{\calF}_m)^2b^2p+4N\}\\
&=(N-4)^2-4({\calL}_m{\calF}_n-2{\calF}_m)^2b^2p\\
&={\calF}_m^2b^4p^2({\calL}_m{\calF}_n-2{\calF}_m)^2-4({\calL}_m{\calF}_n-2{\calF}_m)^2b^2p\\
&=({\calF}_m^2b^2p-4)({\calL}_m{\calF}_n-2{\calF}_m)^2b^2p\\
&={\calL}_m^2({\calL}_m{\calF}_n-2{\calF}_m)^2b^2p,
\end{align*}
as desired.
\qed
\end{proof}

\begin{remark}\label{rem:f}
From the proof of Lemma~\ref{lem:d2}, we have
\begin{align*}
f_{\alpha}(X) &:=X^4-TX^3+({N}+2)X^2-TX+1\\
&=X^4-\L_n\L_mX^3+(\F_mb^2p(\L_m\F_n-2\F_m)+6)X^2-\L_n\L_mX+1,
\end{align*}
for odd integers $m$ and $n$.
\end{remark}

\begin{proposition}\label{prop:pair}
For any odd integers $m,n$ with $n>3$, we have
\[(K,K')=\begin{cases}
(\Q(\sqrt{D_{m,n}}),\Q(\sqrt{pD_{m,n}}))&\text{if $n\equiv 1\mod{4}$},\\
(\Q(\sqrt{pD_{m,n}}),\Q(\sqrt{D_{m,n}}))&\text{if $n\equiv 3\mod{4}$},
\end{cases}\]
where $D_{m,n}$ is defined as in \S1.
\end{proposition}
\begin{proof}
By (\ref{eq:normunit}), we have
\begin{align*}
\alpha^2-4
&=\frac{1}{4} \{ {\mathcal L}_n {\mathcal L}_m+({\mathcal L}_m {\mathcal F}_n-2{\mathcal F}_m) b\sqrt{p} \}^2-4\\
&=\frac{1}{4} \{ {\mathcal L}_n^2{\mathcal L}_m^2+({\mathcal L}_m{\mathcal F}_n-2{\mathcal F}_m)^2 b^2p +2 {\mathcal L}_n {\mathcal L}_m({\mathcal L}_m{\mathcal F}_n-2{\mathcal F}_m)b\sqrt{p} \}-4\\
&=\frac{1}{4} \{ (b^2p {\mathcal F}_n^2 -4){\mathcal L}_m^2+({\mathcal L}_m{\mathcal F}_n-2{\mathcal F}_m)^2 b^2p +2 {\mathcal L}_n{\mathcal L}_m ({\mathcal L}_m{\mathcal F}_n-2{\mathcal F}_m)b\sqrt{p} \}-4 \\
&=\frac{1}{4} \{ 2b^2p {\mathcal F}_n^2 {\mathcal L}_m^2-4{\mathcal L}_m{\mathcal F}_n{\mathcal F}_m b^2p+2 {\mathcal L}_n {\mathcal L}_m({\mathcal L}_m{\mathcal F}_n-2{\mathcal F}_m
)b\sqrt{p} -4({\mathcal L}_m^2-b^2p{\mathcal F}_m^2+4)\} \\
&=\frac{1}{2} \{ b^2p {\mathcal F}_n^2 {\mathcal L}_m^2-2{\mathcal L}_m{\mathcal F}_n{\mathcal F}_m b^2p+ {\mathcal L}_n {\mathcal L}_m({\mathcal L}_m{\mathcal F}_n-2{\mathcal F}_m
)b\sqrt{p} \}.
\end{align*}
Then we have
\begin{align}\label{(12)}
 {\rm Tr}_{k/\Q} \{ (\alpha^2-4)(p\pm \sqrt{p}A)\}  
&=b^2p^2 {\mathcal F}_n^2\calL_m^2-2\calL_m {\mathcal F}_n {\mathcal F}_m b^2p^2 \pm\calL_n\calL_mbp(\calL_m{\mathcal F}_n-2{\mathcal F}_m)A  \\
&=b^2p^2 {\mathcal F}_n\calL_m ({\mathcal F}_n\calL_m-2{\mathcal F}_m)\pm\calL_n\calL_m bp(\calL_m {\mathcal F}_n-2{\mathcal F}_m)A \notag\\
&=({\mathcal F}_n\calL_m-2{\mathcal F}_m)bp\calL_m(bp{\mathcal F}_n\pm\calL_nA).\notag
\end{align}
On the other hand, it follows from Lemma~\ref{lem:d2} that
\[
(\alpha^2-4)(\overline{\alpha}^2-4)=(N+4)^2-4T^2
=\calL_m^2 b^2p (\calL_m {\mathcal F}_n-2{\mathcal F}_m)^2.
\]
Here we recall
$$
b\calL_m(\calL_m {\mathcal F}_n-2{\mathcal F}_m) > 0,
$$
as we have seen in \S{1}. Then we have
\begin{equation*}
\sqrt{(\alpha^2-4)(\overline{\alpha}^2-4)}=\calL_mb\sqrt{p}(\calL_m\calF_n-2\calF_m).
\end{equation*}
From this together with (\ref{(12)}) and Lemma~\ref{lem:abc} (1), (3), we have
\begin{align}
\lambda^2 & = \{ (\varepsilon-\varepsilon^{-1})(S_0-S_2)+(\eta-\eta^{-1})(S_1-S_3) \}^2
\label{eq:lambda}\\
& =-\frac{1}{2} {\rm Tr}_{k/\Q}
\{(\alpha^2-4)(p-\sqrt{p} A) \}+\sqrt{p}B \sqrt{(\alpha^2-4)(\overline{\alpha}^2-4)}\notag\\
& =-\frac{1}{2} ({\calF}_n {\calL}_m -2{\calF}_m)bp{\calL}_m(bp {\calF}_n-{\calL}_nA)+{\calL}_m bpB
({\calL}_m {\calF}_n-2{\calF}_m)\notag\\
& =-\frac{1}{2} ({\calF}_n {\calL}_m-2{\calF}_m)bp{\calL}_m(bp{\calF}_n-{\calL}_nA-2B).\notag
\end{align}
By using Lemma~\ref{lem:abc} (2), (3), we obtain
\begin{equation}\label{eq:mu}
\mu^2=-\frac{1}{2} ({\calF}_n {\calL}_m-2{\calF}_m)bp{\calL}_m(bp{\calF}_n+{\calL}_nA-2B)
\end{equation}
similarly.

Assume that $n\equiv 1\mod{4}$ (resp.\ $n\equiv -1\mod{4}$).
Then by Lemmas~\ref{lem:integer}, \ref{lem:last} and (\ref{eq:lambda}) (resp.\ (\ref{eq:mu})),
we have $\lambda^2\in D_{m,n}\Q^2$ (resp.\ $\mu^2\in D_{m,n}\Q^2$).
Hence $\sqrt{D_{m,n}}\in K$ (resp.\ $\sqrt{D_{m,n}}\in K'$) by Lemma~\ref{lem:lminK}. 
On the other hand, we have $\sqrt{D_{m,n}}\not\in\Q$ because of $D_{m,n}<0$.
Thus we get $K=\Q(\sqrt{D_{m,n}})$ (resp.\ $K'=\Q(\sqrt{D_{m,n}})$).
\qed
\end{proof}
\section{Proof of Main Theorem~\ref{main:1}}\label{2-conditions}

Let the notations be as in \S\ref{sect:main}. 
Namely, we consider the polynomial $f_{\alpha}(X)$ for
$\alpha=\alpha(m,n)$.
Before the proof of Main Theorem~\ref{main:1},
we show the following three lemmas.

\begin{lemma}\label{lem:A3mod}
Assume that two odd integers $m,n$ satisfy 
$({\mathcal L}_m {\mathcal F}_n-2{\mathcal F}_m)b \equiv 0\mod{p^2}$.
Then there exists $x\in \widetilde{L}^{\times}$ such that
\[x^p \equiv \varepsilon^{t(K_0)}
\mod{p(\zeta_p -1) {\mathcal O}_{\widetilde{L}} },\]
that is, $\mathrm{(\ref{A3})}$ holds.
\end{lemma}

\begin{proof}
We get the assertion from  Lemmas~\ref{lem:A3} and \ref{lem:d2}.
\qed
\end{proof}
\begin{lemma}\label{lem:pth}
Let $i,j$ be integers which are not divisible by $p$.
If $\varepsilon^i \eta^j \in L^p$, then we have $\varepsilon,\eta \in L^p$.
\end{lemma}

\begin{proof}
Let $k_1$ be the subfield $\mathbb Q(\zeta)$ of degree $4$. We denote
\[
\Gal(Lk_1/k) \simeq  \langle \sigma \rangle \times \langle
\sigma' \rangle \ (\simeq C_2 \times C_2),
\]
where
$\varepsilon^{\sigma}=\varepsilon^{-1}$, $\eta^{\sigma}=\eta$,
$\varepsilon^{\sigma'}=\varepsilon$ and $\eta^{\sigma'}=\eta^{-1}$.
If $\varepsilon^i \eta^j \in L^p$, then so are
$(\varepsilon^i \eta^j)^{\sigma}=\varepsilon^{-i} \eta^j$, their ratio
$\varepsilon^{2i}$ and their product $\eta^{2j}$. Since $\gcd(2i,p)=\gcd(2j,p)=1$,
we conclude that both $\varepsilon$ and $\eta$ are $p$th powers in $L$.
\qed
\end{proof}
\begin{lemma}\label{lem:pth2}
If $\varepsilon, \eta \not\in L^p$,
then we have $\varepsilon^{t'(K_0)} \not\in \widetilde{L}^p$ for any $t'(K_0)\in T(K_0)$.
\end{lemma}
\begin{proof}
It is sufficient to show that $\varepsilon^{t(K_0)}\not\in \widetilde{L}^p$. Since
$$
\varepsilon^{t(K_0)}=\varepsilon^{\iota_0^3} \eta^{\iota_0^2} \varepsilon^{-\iota_0} \eta^{-1}=
\varepsilon^{\iota_0(\iota_0^2-1)} \eta^{\iota_0^2-1}
$$
and
$$
\iota_0^2-1 =\iota^{\frac{p-1}{2}}-1 \equiv -2 \not\equiv 0\mod{p},
$$
it holds from Lemma~\ref{lem:pth} that $\varepsilon^{t(K_0)}\not\in L^p$.
Then by $p\nmid [\widetilde{L}:L]$, we get $\varepsilon^{t(K_0)}\not\in \widetilde{L}^p$.
\qed
\end{proof}

\begin{proof}{\it of Main Theorem~\ref{main:1}}\ \ 
Let $m_0,n_0$ be integers and $q$ a prime number satisfying the conditions (i), (ii)
in  Main Theorem~\ref{main:1}, and let
\[
(m,n)\in{\mathcal N}:=\{(m,n)\in\Z^2 \,|\, m\equiv m_0 \pmod{N_q},\
n\equiv n_0 \pmod{N_q},\ n>3\}.
\]
Since $m_0\equiv n_0\equiv 1\pmod{2}$
and $N_q$ is even, both $m$ and $n$ are odd.
It holds that 
\[
\calL_m(\calL_m\calF_n-2\calF_m)>0,
\]
as we have stated in \S1.
Then by $\calL_n>0$, both
$\calL_n\calL_m$ and $(\calL_m\calF_n-2\calF_m)b\sqrt{p}$ have the same signs.
Hence by
$$|\calL_n\calL_m|\geq |\calL_5\calL_m|=|(t^5+5t^3+5t)\calL_m|\geq 11,$$
it holds that
$$
|\alpha|=\frac{|\calL_n\calL_m|+|\calL_m\calF_n-2\calF_m|b\sqrt p}{2}
\geq\frac{11}{2}>2.
$$
Thus we obtain $\alpha^2-4>0$.
From this together with Lemma~\ref{lem:d2},
it follows that $\alpha$ satisfies (\ref{A1}).
Moreover, we see from Lemma~\ref{lem:modp^2} that
\[
\F_m\equiv \F_{m_0}, \quad \L_m\equiv \L_{m_0},
\quad
\F_n\equiv \F_{n_0} \pmod{p^2},
\]
hence by Lemma~\ref{lem:A3mod},
a root $\varepsilon$ of $f_{\alpha}(X)$ satisfies (\ref{A3}).

Next, let us prove that the condition (\ref{A2}) holds.
Let $d$ be the discriminant of the characteristic polynomial $P(X)=X^2-tX-1$. Then we have
$d=t^2+4=b^2p$.
It is known (\cite[pp.65--66]{R}) that the periods of 
$\{ \F_n \} \bmod{q}$ and $\{ \L_n \} \bmod{q}$ divide $q-1$ (resp.\ $2(q+1)$)
if $\left( \frac{d}{q} \right)=1$ $\left(\text{resp.}\ \left( \frac{d}{q}\right)=-1 \right)$.
Since $q \nmid 2bp$, we get
$\left(\frac{d}{q}\right)=\left( \frac{b^2p}{q} \right)=\left(\frac{p}{q} \right)$.
By the definition of $N_q$, we have
\[
\F_m\equiv \F_{m_0}, \quad \L_m\equiv \L_{m_0},
\quad
\F_n\equiv \F_{n_0}, \quad \L_n\equiv \L_{n_0} \pmod{q},
\]
and therefore $f_{\alpha,q}(X)=f_{\alpha_0,q}(X) \ (\in \mathbb F_q[X])$.
By the assumption (ii) of Main Theorem~\ref{main:1}, 
we have $f_{\alpha,q}(a)=f_{\alpha_0,q}(a)=0$
for some $i\in \{1,2,4 \}$ and $a\in \mathbb F_{q^i} \setminus \mathbb F_{q^i}^p$.
If $p\nmid q^i-1$, then we have $\mathbb F_{q^i}^p=\mathbb F_{q^i}$ and
this is a contradiction because $a\in \mathbb F_{q^i} \setminus \mathbb F_{q^i}^p$.
We get $p \mid q^i-1$.
Now, we assume that one of $\varepsilon, \varepsilon^{-1},\eta,\eta^{-1}$ 
(hence all of $\varepsilon, \varepsilon^{-1},\eta,\eta^{-1}$) is contained in $L$.
Then we have $a\in \mathbb F_{q^f}^p$ where $f:=[{\mathcal O}_L/{\mathcal Q} :\Z/q\Z]$
for a prime ideal $\mathcal Q$ of $L$ above $q$. If $i\geq f$, then this is a
contradiction because $a\not\in \mathbb F_{q^i}^p$. In the case $i<f$,
we write $a=b^p$ for some $b\in \mathbb F_{q^f}$.
We get $a^{f/i}=N_{\mathbb F_{q^f}/\mathbb F_{q^i}} (b)^p$.
Since $f/i \in \{2,4\}$ and $p\mid q^i-1$, this implies $a\in \mathbb F_{q^i}^p$
and it is a contradiction.
Thus none of $\varepsilon, \varepsilon^{-1},\eta,\eta^{-1}$ is contained in $L$.
By Lemma~\ref{lem:pth2}, therefore, (\ref{A2}) holds.

As for the infiniteness of the set $\{ (k_0(\sqrt{D_{m,n}}), 
k_0(\sqrt{pD_{m,n}})) \,|\, (m,n)\in {\mathcal N} \}$,
it is enough to prove that the set of pairs $$\{ (\Q(\sqrt{D_{m_0,n}}), 
\Q(\sqrt{pD_{m_0,n}}))\,|\,  n\equiv n_0 \pmod{N_q} ,\ n>3 \}$$ is infinite.
For an integer $a$, let $s(a)$ denote  the square free integer satisfying $a= s(a) A^2$ for
some $A\in \N$, and assume that the set  $$\{ (\Q(\sqrt{D_{m_0,n}}), 
\Q(\sqrt{pD_{m_0,n}})) \,|\, n\equiv n_0 \pmod{N_q} ,\ n>3 \}$$ is finite.
Then the set $\{ s(D_{m_0,n}) \,|\, n\equiv n_0 \pmod{N_q},\ n>3 \}$ is finite.
Since there are infinitely many integers $n$ satisfying $n\equiv n_0 \pmod{N_q}$ and $n>3$,
there exists an integer $\ell$ such that 
${\mathcal N}_{\ell} :=\{ n\in \Z \,|\, n\equiv n_0 \pmod{N_q},\ n>3,\ s(D_{m_0,n})=\ell \}$ is
infinite.
For any integer $n\in {\mathcal N}_{\ell}$, let $D_{m_0,n}=\ell A_n^2$.
Then by (\ref{eq:normunit}), we have
\begin{align*}
\L_{m_0}^4 \L_n^2 &=\L_{m_0}^4(b^2p \F_n^2 -4) \\
&=p(\L_{m_0}^2 b\F_n)^2-4 \L_{m_0}^4 \\
&=p(2b\F_{m_0}\L_{m_0}-\ell A_n^2)^2 -4\L_{m_0}^4\\
&=p\ell^2 A_n^4-4bp\ell \F_{m_0}\L_{m_0}A_n^2 +4b^2p \F_{m_0}^2 \L_{m_0}^2
-4\L_{m_0}^4.
\end{align*}
This implies that infinitely many pairs $(A_n,\L_n)$ are integer solutions of the equation
\begin{equation*}
\L_{m_0}^4 Y^2=p\ell^2 X^4 -4bp\ell \F_{m_0}\L_{m_0}X^2+4b^2p\F_{m_0}^2 \L_{m_0}^2-4\L_{m_0}^4.
\end{equation*}
The discriminant of the quartic polynomial on the right side is 
\[
2^{14} p^3 \ell^6 \L_{m_0}^{10}(b^2p\F_{m_0}^2-\L_{m_0}^2) =2^{16} p^3 \ell^6 \L_{m_0}^{10} \ne 0,
\]
by (\ref{eq:normunit}) and the assumption $m_0\equiv 1 \pmod{2}$.
Hence the equation has only finitely many integer solutions by Siegel's theorem.
This is a contradiction, and
the proof is complete.
\qed
\end{proof}

\section{Proof of Main Theorem~\ref{main:2}}\label{sec:MT2}

In this section, we prove Main Theorem~\ref{main:2}.
Let $q\,(\ne 2)$ be a prime number and $\mathbb F_{q^r}$ be the
finite field with the cardinality $q^r$.
We denote by $g$ a generator of the cyclic group $\mathbb F_{q^r}^{\times}$.
Put
\[
Y_q:=\{ (g^m-g^{-m})g^n-(g^m+g^{-m})\,|\,n,m\in\Z,\ n\equiv m\equiv 1\pmod{2} \}.
\]
The set $Y_q$ does not depend on $g$ because other generators are given by
$g^s$ with $(s,q^r-1)=1$. First, we show the following lemma.

\begin{lemma}\label{lem:Y}
Let $q\,(\ne 2)$ be a prime number with $q^r>45$. Then we have $Y_q=\mathbb F_{q^r}$.
\end{lemma}

\begin{proof}
Put $k=(g^m-g^{-m})g^n-(g^m+g^{-m})$, $m=2u+1,\ n=2v+1 \ (u,v\in \Z)$ and
$X=g^u$, $Y=g^v$.
Then we have
\begin{equation*}
f(X,Y):=
g^3X^4Y^2-gY^2-g^2X^4-kgX^2-1=0.
\end{equation*}
By the definition of $Y_q$, we easily see that $Y_q=\mathbb F_{q^r}$ if and only if
\begin{equation*}
S_k:=\{ (X,Y)\in \mathbb F_{q^r}^2  \,|\,  f(X,Y)=0,\ XY\ne 0 \} \ne \emptyset
\end{equation*}
for any $k\in \mathbb F_{q^r}$.
Because if $S_k \ne \emptyset$ for $k\in \mathbb F_{q^r}$, then there exist integers $u,v$ satisfying
$f(g^u,g^v)=0$. This implies 
\[
g^{2m}g^n-g^n-g^{2m}-kg^m-1=0,
\]
where $m=2u+1$ and $n=2v+1$, and we get 
\[
k=(g^m-g^{-m})g^n -(g^m+g^{-m}),
\]
and hence $k\in Y_q$.

(i) Consider the case $k\ne \pm 2$.
By putting $Y=Z/(g^3X^4-g)$, we get
\begin{equation*}
f(X,Y)=-(g^3X^4-g)^{-1} (g^5X^8+g^4kX^6-g^2kX^2-g-Z^2).
\end{equation*}
Put
\begin{equation*}
C_k: Z^2=g(X)
\end{equation*}
with
\begin{equation*}
g(X)=g^5X^8+g^4kX^6-g^2kX^2-g=(g^3X^4-g)(1+gkX^2+g^2X^4).
\end{equation*}
Since $g$ is a generator of $\mathbb F_{q^r}^{\times}$, if there exists
$X_0\in \mathbb F_{q^r}$ satisfying $$g^3X_0^4-g=g(gX_0^2+1)(gX_0^2-1)=0,$$  then we get
$gX_0^2+1=0$. For such an $X_0$ and any $Y\in \mathbb F_{q^r}$, we have
$f(X_0,Y)=k-2\ne0$.
Therefore, for any $(X,Y)\in \mathbb F_{q^r}^2$ such that $f(X,Y)=0$, we have
$g^3X^4-g\ne 0$.
We conclude that there is one-to-one correspondence between the sets $S_k$
and $\{ (X,Z)\in C_k(\mathbb F_{q^r}) \,|\, XZ\ne 0 \}$ by $(X,Y) \mapsto
(X,Y(g^3X^4-g))$.
Since $k\ne \pm2$, we have $C_k$ is a smooth (hyperelliptic) curve of genus 3 with the
discriminant
$-2^{12}g^{42}(k-2)^6(k+2)^6$. Let $\widetilde{C}_k$ be the smooth projective curve by
adding two infinite points. Since the leading coefficient $g^5$ of $g(X)$ is not a square,
these infinite points are not rational, and hence we get
$\widetilde{C}_k(\mathbb F_{q^r})=C_k(\mathbb F_{q^r})$.
By a consequence of  Weil's theorem, we have
\begin{equation*}
\sharp C_k(\mathbb F_{q^r})=\sharp \widetilde{C}_k (\mathbb F_{q^r}) \geq q^r+1-6\sqrt{q^r}.
\end{equation*}
Since
\begin{align*}
\sharp\{(0,Z)\in C_k(\mathbb F_{q^r})\}&=\sharp\{Z\in \mathbb F_{q^r}\,|\,Z^2+g=0\}\leq 2,\\
\sharp\{(X,0)\in C_k(\mathbb F_{q^r})\}&=\sharp\{X\in \mathbb F_{q^r}\,|\,1+gkX^2+g^2X^4=0\}\leq 4,
\end{align*}
we have
\begin{equation*}
\sharp\{(X,Z)\in C_k(\mathbb F_{q^r})\,|\, XZ\ne 0 \}\geq q^r+1-6\sqrt{q^r}-6,
\end{equation*}
and hence $\{(X,Z)\in C_k(\mathbb F_{q^r}) \,|\, XZ\ne 0\}  \ne \emptyset$ if $q^r >45$.
We conclude that $S_k \ne \emptyset$ if $q^r>45$.

(ii) Consider the case $k=2$. We note that
\[
f(X,Y)=(gX^2+1)(g^2X^2Y^2-gX^2-gY^2-1)
\]
in this case.

If $q^r\equiv 3 \pmod{4}$, then we have $gX_0^2+1=0$ for $X_0:=\pm g^{(q^r-3)/4} \in \mathbb F_{q^r}$.
Hence we have
$
f(X_0,Y)=0
$
for any $Y\in \mathbb F_{q^r}$.

If $q^r\equiv 1\pmod{4}$, then we have $g^3X^4-g=g(gX^2-1)(gX^2+1)\ne 0$ for any 
$X\in \mathbb F_{q^r}$. By putting $Y=Z/g(gX^2-1)$, we get
\begin{equation*}
f(X,Y)=-\frac{gX^2+1}{g(gX^2-1)} ((g^3X^4-g)-Z^2).
\end{equation*}
Put 
\begin{equation*}
C_2: Z^2=g^3X^4-g.
\end{equation*}
There is one-to-one correspondence between the sets $S_2$
and $\{ (X,Z) \in C_2(\mathbb F_{q^r}) \,|\, XZ\ne 0\}$ by
$(X,Y) \mapsto (X,g(gX^2-1)Y)$. Since $C_2$ is a smooth curve of genus $1$,
by similar arguments of (i), we have
\begin{equation*}
\sharp C_2(\mathbb F_{q^r})=\sharp \widetilde{C}_2 (\mathbb F_{q^r}) \geq q^r+1-2\sqrt{q^r}.
\end{equation*}
Since
\begin{align*}
\sharp \{(0,Z)\in C_2(\mathbb F_{q^r}) \}&=\sharp \{ Z\in \mathbb F_{q^r} \,|\, Z^2+g=0 \}=0,\\
\sharp \{(X,0)\in C_2(\mathbb F_{q^r}) \}&=\sharp \{ X\in \mathbb F_{q^r} \,|\, g^3X^4-g=0 \}=0,
\end{align*}
we have
\begin{equation*}
\sharp \{(X,Z)\in C_2(\mathbb F_{q^r}) \,|\, XZ\ne 0 \} \geq q^r+1-2\sqrt{q^r}=(\sqrt{q^r}-1)^2>0.
\end{equation*}

We conclude that $S_2 \ne\emptyset$ for any prime number $q$.

(iii) Consider the case $k=-2$.
 By putting $Y=Z/g(gX^2+1)$, we get
\begin{equation*}
f(X,Y)=-\frac{gX^2-1}{g(gX^2+1)} ((g^3X^4-g)-Z^2).
\end{equation*}
Put 
\begin{equation*}
C_{-2}: Z^2=g^3X^4-g.
\end{equation*}

If $q^r\equiv 3 \pmod{4}$, then we have $gX_0^2+1=0$ for $X_0:=\pm g^{(q^r-3)/4} \in \mathbb F_{q^r}$.
Hence we have
\begin{equation*}
f(X_0,Y)=-((gX_0^2-1)^2-Y^2(g^3X_0^4-g))=-(gX_0^2-1)^2=-4 \ne 0,
\end{equation*}
for any $Y\in \mathbb F_{q^r}$. Therefore, for any $(X,Y)\in \mathbb F_{q^r}^2$ such that $f(X,Y)=0$,
we have $gX^2+1 \ne 0$. We conclude that 
there is one-to-one correspondence between the sets $S_{-2}$
and $\{ (X,Z) \in C_{-2}(\mathbb F_{q^r}) \,|\, XZ\ne 0 \}$ by
$(X,Y) \mapsto (X,g(gX^2+1)Y)$. 
In this case, we have
\begin{equation*}
\sharp C_{-2}(\mathbb F_{q^r})=\sharp \widetilde{C}_{-2} (\mathbb F_{q^r}) \geq q^r+1-2\sqrt{q^r},
\end{equation*}
and
\begin{align*}
\sharp\{(0,Z)\in C_{-2}(\mathbb F_{q^r})\}&=\sharp\{Z\in\mathbb F_{q^r}\,|\,Z^2+g=0\}=2,\\
\sharp\{(X,0)\in C_{-2}(\mathbb F_{q^r})\}&=\sharp\{X\in\mathbb F_{q^r}\,|\,g^3X^4-g=0\}=2,
\end{align*}
and hence 
\begin{equation*}
\sharp\{(X,Z)\in C_{-2}(\mathbb F_{q^r})\,|\,XZ\ne 0 \}\geq q^r+1-2\sqrt{q^r}-4.
\end{equation*}
Thus we have $\{(X,Z)\in C_{-2}(\mathbb F_{q^r})\,|\,XZ\ne 0 \}\ne\emptyset$ if $q^r>9$.

If $q^r\equiv 1\pmod{4}$, then we have $gX^2+1\ne 0$ for any $X\in \mathbb F_{q^r}$.
By the same argument of (ii) in the case $q^r \equiv 1 \pmod{4}$, we have
\begin{equation*}
\sharp S_{-2}=\sharp\{(X,Z)\in C_{-2} (\mathbb F_{q^r}) \,|\, XZ\ne 0  \}>0.
\end{equation*}

We conclude that $S_{-2}\ne\emptyset$ if $q^r>9$.

By (i), (ii) and (iii), we conclude that $Y_q=\mathbb F_{q^r}$ for any prime number $q$ with $q^r>45$.
\qed
\end{proof}

\begin{proposition}\label{prop:GRH}
Assume that ERH holds. Then there exists odd integers $m,n$ and a prime number $q$
such that $q\nmid 2bp$, $q^2\not\equiv 1\pmod{p}$ and $f_{\alpha,q}(a)=0$ for some
$a\in \mathbb F_{q^f} \setminus \mathbb F_{q^f}^p$, where $\alpha:=\alpha(m,n)$ and
$f:=[{\mathcal O}_L/{\mathcal Q} :\Z/q\Z]$ for a prime ideal $\mathcal Q$ of $L$ above $q$. 
\end{proposition}

\begin{proof}
We use a result proved by Lenstra~\cite[(4.8)]{L} for $k=\Q (\sqrt{p})$. Let
$\sigma$ and $\sigma'$ be generators of the cyclic groups
$\Gal(\widetilde{L}/L)\, (\simeq C_{(p-1)/2})$
and $\Gal(\widetilde{L}/\Q(\zeta_p))\, (\simeq C_2)$, respectively, and put
$\tau:= \sigma^{(p-1)/4}\sigma'$.
Consider the set $M=M(k,\widetilde{L},\{\tau\},\langle u_p \rangle, 1)$ of primes 
$\mathfrak q$ of $k$ satisfying $(\mathfrak q, \widetilde{L}/k)=\tau$ and
$({\mathcal O}_k/{\mathfrak q})^{\times} =\langle u_p \bmod{\mathfrak q} \rangle$
(see \cite[p.203]{L}). Let $\ell$ be a prime number, and assume
$L_{\ell}:=\Q(\zeta_{\ell},\sqrt[\ell]{u_p})\subset \widetilde{L}$
and $\tau \in \Gal(\widetilde{L}/L_{\ell})$.
By the definition of $\tau$, the fixed field of $\langle \tau \rangle$ coincides with 
$K(\omega)=K'(\omega)$.
We have $L_{\ell} \subset K(\omega)=K'(\omega)$. Since the field $K(\omega)$ is abelian extension over $\Q$,
$L_{\ell}/\Q$ is also abelian extension.
This is a contradiction, because we have
\[
\Q \subset k=\Q (u_p) \subset \Q(\sqrt[\ell]{u_p}) \subset L_{\ell},
\]
but  $\Q (\sqrt[\ell]{u_p})/k$ is not a Galois extension for any prime number $\ell \geq 3$,
and $\Q(\sqrt{u_p})/\Q$ is not a Galois extension since the Galois conjugate $\overline{u}_p$ satisfies
$\overline{u}_p=-1/u_p$ (see the beginning of \S \ref{sect:F-L}) and 
$\sqrt{-1/u_p} \not\in \Q(\sqrt{u_p})$.
We conclude that there is no prime number $\ell$ satisfying $L_{\ell} \subset \widetilde{L}$ and
$\tau \in \Gal(\widetilde{L}/L_{\ell})$.
By Lenstra's result (\cite[(4.8)]{L}),  the set $M$ is infinite.
Choose ${\mathfrak q}\in M$ which is unramified in $\widetilde{L}/k$ and satisfies
$q \nmid 2bp,\ q>45$ for the prime number $q$ such that ${\mathfrak q}\mid q$.
Since $({\mathfrak q},\Q(\zeta_p)/k)$ is the restriction of
$\sigma^{\frac{p-1}{4}}\in\Gal(\widetilde{L}/L)$ to $\Q(\zeta_p)$ and $({\mathfrak q},L/k)$
is the restriction of $\sigma'\in\Gal(\widetilde{L}/\Q(\zeta_p))$
to $L$, we see that ${\mathfrak q}$ is totally decomposed in $\Q(\zeta_p+\zeta_p^{-1})/k$ and not
decomposed in both $\Q(\zeta_p)/\Q(\zeta_p+\zeta_p^{-1})$ and $L/k$.
Put $r:=[{\mathcal O}_k/{\mathfrak q} :\Z/q\Z]$ and $f:=[{\mathcal O}_L/{\mathcal Q}:\Z/q\Z]$.
Then we have $f=2r$ and the order of $q$ in $\mathbb F_p^{\times}$ is $2r$
(hence, $q^f=q^{2r}\equiv 1 \pmod{p}$, $q^2\not\equiv 1\pmod{p}$ and $q^r\not\equiv 1\pmod{p}$).
On the other hand, since $u_p\overline{u}_p=-1$, for  odd integers $m,n$, we have
\begin{align*}
\alpha (m,n) &= \frac{\L_n \L_m+(\L_m\F_n-2\F_m)b\sqrt{p}}{2} \\
&=\L_m u_p^n -\F_m(u_p-\overline{u}_p)\\
& =(u_p^m+\overline{u}_p^m)u_p^n -(u_p^m-\overline{u}_p^m)\\
& =(u_p^m-u_p^{-m}) u_p^n-(u_p^m+u_p^{-m}).
\end{align*}
Since $\langle u_p \bmod{\mathfrak q} \rangle =({\mathcal O}_k/{\mathfrak q})^{\times}
\simeq \mathbb F_{q^r}^{\times}$, 
$q^r\geq q>45$, by  Lemma~\ref{lem:Y}, we get
\begin{equation}\label{eq:all}
\{ \alpha=\alpha(m,n) \bmod{\mathfrak q} \in {\mathcal O}_k/{\mathfrak q}\,|\, n\equiv m\equiv 1 \pmod{2} \}
={\mathcal O}_k/{\mathfrak q}.
\end{equation}
From (\ref{eq:all}) and
\[
f_{\alpha}(X) \equiv (X^2-\alpha X+1)(X^2-\overline{\alpha}X+1) \pmod{\mathfrak q},
\]
it is enough to show 
\[
\{\beta \in \mathbb F_{q^r}\,|\, a^2-\beta a+1 =0 \ \text{for\ some}\ a\in \mathbb F_{q^f} \setminus
\mathbb F_{q^f}^p \} \ne \emptyset.
\]
Put $\mathbb F_{q^f}^{\times} =\langle g\rangle $ and $x_s:=g^s$ for $s\in \{1,2,\ldots,q^f-1 \}$.
Since $\Gal(\mathbb F_{q^f}/{\mathbb F}_{q^r})$ is generated by ${\rm Frob}_q$
which is defined by $x^{{\rm Frob}_q}=x^{q^r}$ for any $x\in \mathbb F_{q^f}$, we have
\begin{align*}
N_{\mathbb F_{q^f}/\mathbb F_{q^r}} (x_s)=1 & \Longleftrightarrow x_s^{1+q^r}=1\\
& \Longleftrightarrow g^{s(1+q^r)}=1\\
& \Longleftrightarrow s=(q^r-1)u,\ u\in \{1,2,\ldots,q^r+1\}.
\end{align*}
Therefore $N_{\mathbb F_{q^f}/\mathbb F_{q^r}}(x_s)=1$ and $x_s\not\in \mathbb F_{q^f}^p$ if and only if
$s=(q^r-1)u,\ u\in \{1,2,\ldots, q^r+1 \}$ and $p\nmid u$.
Since $q^r\not\equiv 1\pmod{p}$, we have $\mathbb F_{q^r}=\mathbb F_{q^r}^p$. Hence if $x_s \not\in
\mathbb F_{q^f}^p$, then $x_s\not\in \mathbb F_{q^r}$, and both $x_s$ and $x_{sq^r}$ have the same
minimal polynomial over $\mathbb F_{q^r}$.
Therefore, we conclude 
\begin{align*}
\sharp
\{\beta\in \mathbb F_{q^r}\,|\, a^2-\beta a+1 =0\
\text{for\ some}\ a\in \mathbb F_{q^f} \setminus
\mathbb F_{q^f}^p \} 
& =\frac{1}{2} \left( q^r+1-\frac{q^r+1}{p} \right)\\
& =\frac{1}{2p}(q^r+1)(p-1) >0,
\end{align*}
and the proof is complete.
\qed
\end{proof}
\begin{lemma}\label{lem:divisorF}
If $p^{\nu} \mid n$, then $p^{\nu} \mid \F_n$.
\end{lemma}
\begin{proof}
Since
\begin{align*}
u_p-\overline{u}_p &=b\sqrt{p},\\
u_p^n-\overline{u}_p^n & =2^{-n} \{ (t+b\sqrt{p})^n -(t-b\sqrt{p})^n \}\\
& \equiv 2^{-n} \{ (t^n+nt^{n-1} b\sqrt{p})-(t^n-nt^{n-1}b\sqrt{p}) \}\\
& \equiv 2^{-n+1} nt^{n-1}b\sqrt{p} 
\pmod{p^{ \nu+1} b{\mathcal O}_k},
\end{align*}
we have
\[
\F_n=\frac{u_p^n-\overline{u}_p^n}{u_p-\overline{u_p}} \equiv 2^{-n+1}nt^{n-1} \equiv 0
\pmod{p^{\nu} {\mathcal O}_k}.
\]
Therefore, we have $p^{\nu} \mid \F_n$.
\qed
\end{proof}
\begin{proof}{\it of Main Theorem~\ref{main:2}}\ \
Under the ERH, there exist odd integers $m,n$ and a prime number $q$ satisfying
the conditions in Proposition~\ref{prop:GRH}.
Since $q^2\not\equiv 1\pmod{p}$, there exists $c\in \Z$ such that $p^2c\equiv 1\pmod{2(q^2-1)}$.
Put $m_0:=p^2cm$ and $n_0:=p^2cn$. Then both $m_0$ and $n_0$ are odd.
We prove that $m_0, n_0$ satisfy the conditions (i), (ii) of Main Theorem~\ref{main:1}.
Since $p^2\mid m_0,\ p^2\mid n_0$, we get $p^2\mid \F_{m_0},\ p^2 \mid \F_{n_0}$
by Lemma~\ref{lem:divisorF}, and hence $m_0,n_0$ satisfy the condition (i).
Since the periods of $\{ \F_n \} \bmod{q}$ and $\{ \L_n \} \bmod{q}$ divide
$2(q^2-1)$ (\cite[pp.65--66]{R}), and $m_0=p^2cm\equiv m\pmod{2(q^2-1)}$ and
$n_0=p^2cn \equiv n\pmod{2(q^2-1)}$, we have
\[
\F_{m_0} \equiv \F_m, \quad \L_{m_0}\equiv \L_m,
\quad
\F_{n_0} \equiv \F_n, \quad \L_{n_0}\equiv \L_n \pmod{q}.
\]
Therefore $f_{\alpha_0,q}(X)=f_{\alpha,q}(X) \ (\in \mathbb F_q[X])$
for $\alpha_0:=\alpha(m_0,n_0)$ and $\alpha:=\alpha(m,n)$.
We know that $m_0,n_0$ satisfy the condition (ii)
for $i=f$.
The proof is complete.
\qed
\end{proof}

\section{Examples}

\begin{example}\label{ex:p=5}
(1)
Let $p=5$. Then the fundamental unit of $k$ is $u_p=(1+\sqrt{5})/2$,
and hence $t=b=1$. So the sequences $\{ \F_n \}$ and $\{ \L_n \}$ are the same as
the Fibonacci numbers $\{ F_n \}$ and the Lucas numbers $\{ L_n \}$, respectively.
Now we will verify that
any pair of integers  $m_0$ and  $n_0$ in  Table \ref{table:1} and a prime number $q=11$
satisfy the conditions (i), (ii) of the Main Theorem~\ref{main:1}.
\begin{table}[H]
\caption{$p=5,q=11$}
\label{table:1}
\begin{center}
\begin{tabular}{c||c|c|c|c|c}
$m_0 \bmod{50}$ & 7 & 17 & 27 & 37 & 47 \\ \hline
$n_0\bmod{100}$ & 31 & 11 & 91 & 71 & 51 
\end{tabular}
\end{center}
\end{table}

Since $b=1$, the condition (i) in the Main Theorem~\ref{main:1} is equivalent to
$F_{n_0} \equiv 2F_{m_0} L_{m_0}^{-1}$ $\pmod{5^2}$
(Note that $p\nmid L_{m_0}$ from (\ref{eq:normunit})).
We see that $\{F_{n_0}\} \bmod{5^2}$ is $100$ periodic and $\{2F_{m_0}L_{m_0}^{-1}\} \bmod{ 5^2}$
is $25$ periodic. Hence any pair of integers $m_0$ and $n_0$ in Table \ref{table:1}
satisfy the condition (i) of the Main Theorem~\ref{main:1} from Tables \ref{table:2}
and \ref{table:3}.
\begin{table}[H]
\caption{$2F_{m_0}L_{m_0}^{-1} \bmod{25}$}
\label{table:2}
\begin{center}
\begin{tabular}{c||c|c|c|c|c}
$m_0 \bmod{50}$ & 7 & 17 & 27 & 37 & 47 \\ \hline
$2F_{m_0}L_{m_0}^{-1} \bmod{25}$ & 19 & 14 & 9 & 4 & 24 
\end{tabular}
\end{center}
\end{table}
\begin{table}[H]
\caption{$F_{n_0} \bmod{25}$}
\label{table:3}
\begin{center}
\begin{tabular}{c||c|c|c|c|c}
$n_0 \bmod{100}$  & 11 &  31 &  51 & 71 &  91 \\ \hline
$F_{n_0} \bmod{25}$  & 14 & 19 &  24 & 4 &  9 
\end{tabular}
\end{center}
\end{table}

Next, both $\{F_{n_0}\} \bmod{11}$ and $\{L_{n_0}\} \bmod{11}$ are $10$ periodic.
Since $F_{m_0}\equiv 2 \pmod{11}$, $L_{m_0}\equiv 7\pmod{11}$ for $m_0$ with $m_0\equiv 7\pmod{10}$
and  $F_{n_0}\equiv L_{n_0} \equiv 1 \pmod{11}$ for $n_0$ with $n_0\equiv 1\pmod{10}$,
we have
\begin{align*}
f_{\alpha_{0}}(X) & \equiv X^4+4X^3+ 3X^2+4X+1 \\
& \equiv (X-5)(X-7)(X-8)(X-9) \pmod{11} ,
\end{align*}
and $a:=5,7,8,9 \bmod{11} \not\in (\mathbb F^{\times}_{11})^5=\langle 2^5 \rangle =\{ \pm 1\}$.
Therefore, the condition (ii) holds for $i=1$.

(2)
Let $p=13$. Then the fundamental unit of $k$ is $u_p=(3+\sqrt{13})/2$.
We will verify that
any pair of integers  $m_0$ and $n_0$ in  Table \ref{table:4} and a prime number $q=53$
satisfy the conditions (i), (ii) of the Main Theorem~\ref{main:1}.
\begin{table}[H]
\caption{$p=13,q=53$}
\label{table:4}
\begin{center}
\begin{tabular}{c||c|c|c|c|c|c|c|c|c|c|c|c|c}
$m_0 \bmod{2\times 13^2}$ &  15 & 41 & 67 & 93 & 119 & 145 & 171& 197 & 223 & 249 & 275 & 301 & 327  \\ \hline
$n_0\bmod{2^2\times 13^2}$ &   55  &  263  & 471 & 3 & 211 & 419 & 627& 159 & 367 & 575 & 107 & 315 & 523 
\end{tabular}
\end{center}
\end{table}

Since $b=1$, the condition (i) in the Main Theorem~\ref{main:1} is equivalent to
$\F_{n_0} \equiv 2\F_m \L_{m_0}^{-1}$ $\pmod{13^2}$.
We see that $\{\F_{n_0}\} \bmod{13^2}$ is $676\, (=2^2\times 13^2)$ periodic and
$\{2\F_{m_0}\L_{m_0}^{-1}\} \bmod{ 13^2}$ is $169\, (=13^2)$ periodic.
Hence any pair of integers $m_0$ and $n_0$ in Table \ref{table:4} satisfies the condition (i) of the
Main Theorem~\ref{main:1} from Tables \ref{table:5} and  \ref{table:6}
\begin{table}[H]
\caption{$2\F_{m_0}L_{m_0}^{-1} \bmod{13^2}$}
\label{table:5}
\begin{center}
\begin{tabular}{c||c|c|c|c|c|c|c|c|c|c|c|c|c}
$m_0 \bmod{2\times 13^2}$ & 15 & 41 & 67 & 93 & 119 & 145 & 171& 197& 223 & 249 & 275 & 301 & 327 \\ \hline  
$2\F_{m_0}  \L_{m_0}^{-1} \bmod{13^2}$ & 127 & 88 & 49 & 10 & 140 & 101 & 62 & 23 & 153 & 114 & 75 & 36 & 166
\end{tabular}
\end{center}
\end{table}
\begin{table}[H]
\caption{$\F_{n_0} \bmod{13^2}$}
\label{table:6}
\begin{center}
\begin{tabular}{c||c|c|c|c|c|c|c|c|c|c|c|c|c}
$n_0 \bmod{2^2\times 13^2}$ & 3 & 55 & 107 &
159 &  211 &  263 & 315 & 367  & 419 &   471 &
 523 & 575  & 627\\ \hline
$\F_{n_0}   \bmod{13^2}$ & 10 &127 &  75 & 
 23 & 140  & 88  & 36&  153 & 101 &  49  & 
166 & 114 & 62 
\end{tabular}
\end{center}
\end{table}

Next, both $\{\F_{n_0}\} \bmod{53}$ and $\{\L_{n_0}\} \bmod{53}$ are $26$ periodic.
Since $\F_{m_0}\equiv 24 \pmod{53}$, $\L_{m_0}\equiv 8\pmod{53}$ for $m_0$ with $m_0\equiv 15\pmod{26}$
and  $\F_{n_0}\equiv 10 \pmod{53}$, $\L_{n_0} \equiv 36$ $\pmod{53}$ for $n_0$ with $n_0\equiv 3\pmod{26}$,
we have
\begin{align*}
f_{\alpha_{0}}(X) & \equiv X^4+30X^3+ 26X^2+30X+1 \\
& \equiv (X-22)(X-24)(X-41)(X-42) \pmod{53} ,
\end{align*}
and $a:=22,24,41,42 \bmod{11} \not\in (\mathbb F^{\times}_{53})^{13}=\langle 2^{13} \rangle =\{ 1,23,30,52\}$.
Therefore, the condition (ii) holds for $i=1$.
\end{example}

\begin{example}\label{ex:2}
Main Theorem~\ref{main:1} implies the previous theorem (Theorem~\ref{theo:JNT} in
\S\ref{sect:intro}). Indeed, for $p=5$, we can check that any pairs
$(m_0,n_0) \in \{ (1,97),(1,103),(1,197),(1,203) \}$ and
$q=61$ satisfy the conditions (i), (ii) as follows.

(i) Since $m_0=1$,
we have $\F_{m_0} =F_1=1$ and $\L_{m_0} =L_1=1$. Furthermore, we have
$\F_{n_0}=F_{n_0} \equiv 2\pmod{5^2}$ since $n_0 \equiv \pm 3 \pmod{100}$.
Therefore, the condition (i) holds.

(ii) The polynomials $f_{\alpha_0 ,61} (X) \in \mathbb F_{61} [X]$ for $\alpha_0=\alpha(m_0,n_0)$
are
\begin{equation*}
f_{\alpha_0,61} (X)=\begin{cases}
(X-10)(X-30)(X-55)(X-59) & \text{if}\ (m_0,n_0)=(1,97),\\
(X-26)(X-33)(X-37)(X-54) & \text{if}\ (m_0,n_0)=(1,103),\\
(X-7)(X-24)(X-28)(X-35) & \text{if}\ (m_0,n_0)=(1,197),\\
(X-2)(X-6)(X-31)(X-51) & \text{if}\ (m_0,n_0)=(1,203).
\end{cases}
\end{equation*}
The condition (ii) holds for $i=1$ since
\begin{equation*}
({\mathbb F}_{61}^{\times})^5 =\{1,11,13,14,21,29,32,40,47,48,50,60 \}.
\end{equation*}

Therefore Main Theorem~\ref{main:1} implies that the class numbers
of both imaginary quadratic fields $\Q(\sqrt{2-F_n})$ and $\Q(\sqrt{5(2-F_n)})$
are divisible by $5$ for any
$$n\in \{ n\in \Z \,|\, n\equiv 97,103,197,203 \mod{N_q},\ n>3\}.$$
By the definition of $N_q$ and $q=61 \equiv 1 \pmod{5}$, we have 
$$N_q=\lcm(p^2(p-1),q-1)=300.$$
Then we have
\begin{align*}
\{ n\in \Z \,|\, n\equiv 97, 103,197,203 \mod{N_q} ,\ n>3 \} 
& =\{ n\in \N \,|\, n\equiv \pm 3 \mod{100},\ n\not\equiv 0\mod{3} \} \\
& \supset\{ n\in \N \,|\, n\equiv \pm 3 \mod{500},\ n\not\equiv 0\mod{3} \},
\end{align*}
and hence we get the set of pairs which is given in Theorem~\ref{theo:JNT}.
\end{example}
\begin{acknowledgements}
The authors would like to thank Toru Komatsu for useful advices.
They would also like to thank  Takuya Yamauchi for
his polite suggestions on the proof of Lemma~\ref{lem:Y}.
\end{acknowledgements}

\end{document}